\documentclass[a4paper, 12pt]{article}

\usepackage[top = 2.5cm, bottom = 2.5cm, left = 2.4cm, right = 2.4cm]{geometry}
\usepackage{amsfonts, amsmath, amsthm, amssymb, mathtools, cases, bbm, array, multirow, booktabs}
\usepackage{lmodern, indentfirst, setspace, changepage, fancyhdr}
\usepackage[vlined, ruled, linesnumbered]{algorithm2e}
\usepackage{tikz, pgfplots, float, subcaption}
\usepackage[title]{appendix}
\usepackage[english]{babel}
\usepackage[T1]{fontenc}
\usepackage{comment}
\usepackage{cite}
\usepackage{hhline}

\usepackage{hyperref}
\hypersetup{
colorlinks = true,
linkcolor = blue,
citecolor = blue,
}

\newtheorem{proposition}{Proposition}

\newtheorem{theorem}{Theorem}
\newtheorem{lemma}{Lemma}
\theoremstyle{definition}

\newtheorem{remark}{Remark}

\newcommand{\argmin}[1]{\underset{#1}{\mathrm{argmin}}}
\newcommand{\argmax}[1]{\underset{#1}{\mathrm{argmax}}}
\newcommand{\ind}[1]{\mathbbm{1}_{\left\{#1\right\}}}
\newcommand{\norm}[1]{\left|\left|#1\right|\right|}

\newcommand{\map}[3]{#1 : #2 \longrightarrow #3}

\newcommand{\set}[2]{\left\{#1 : #2\right\}}

\newcommand{\bsigma}{\boldsymbol{\sigma}}

\newcommand{\boldf}{\boldsymbol{f}}
\newcommand{\defeq}{\vcentcolon=}

\newcommand{\sign}{\mathrm{sign}}

\newcommand{\bd}{\boldsymbol{d}}
\newcommand{\bg}{\boldsymbol{g}}

\newcommand{\bi}{\boldsymbol{i}}

\newcommand{\bx}{\boldsymbol{x}}

\newcommand{\bz}{\boldsymbol{z}}

\newcommand{\prob}{\mathbbm{P}}
\newcommand{\calA}{\mathcal{A}}

\newcommand{\calE}{\mathcal{E}}
\newcommand{\calF}{\mathcal{F}}
\newcommand{\calG}{\mathcal{G}}

\newcommand{\calI}{\mathcal{I}}

\newcommand{\calL}{\mathcal{L}}
\newcommand{\calM}{\mathcal{M}}
\newcommand{\calN}{\mathcal{N}}

\newcommand{\calV}{\mathcal{V}}

\newcommand{\calY}{\mathcal{Y}}

\newcommand{\Bin}{\mathrm{Bin}}
\newcommand{\scdot}{{}\cdot{}}
\newcommand{\bark}{{\bar{k}}}

\newcommand{\bars}{{\bar{s}}}

\newcommand{\N}{\mathbbm{N}}

\newcommand{\R}{\mathbbm{R}}

\newcommand{\e}{\mathrm{e}}

\newcommand{\condp}{P\probarg}
\newcommand{\conde}{E\expectarg}

\DeclarePairedDelimiterX{\probarg}[1]{(}{)}{%
\ifnum\currentgrouptype=16 \else\begingroup\fi
\activatebar#1
\ifnum\currentgrouptype=16 \else\endgroup\fi
}
\DeclarePairedDelimiterX{\expectarg}[1]{[}{]}{%
\ifnum\currentgrouptype=16 \else\begingroup\fi
\activatebar#1
\ifnum\currentgrouptype=16 \else\endgroup\fi
}
\newcommand{\innermid}{\nonscript\;\delimsize\vert\nonscript\;}
\newcommand{\activatebar}{%
\begingroup\lccode`\~=`\|
\lowercase{\endgroup\let~}\innermid 
\mathcode`|=\string"8000
}

\usetikzlibrary{automata, arrows, positioning, calc, external, babel, backgrounds, matrix, shapes}
\usepgfplotslibrary{fillbetween}
\pgfplotsset{
compat = 1.16,
ticklabel style = {font = \footnotesize},
every axis/.append style = {
grid style = {dashed, gray, opacity = 0.2},
label style = {font = \footnotesize}, 
width = 0.6 * \columnwidth,
height = 0.618 * 0.6 * \columnwidth
}
}

\definecolor{britishracinggreen}{rgb}{0.0, 0.26, 0.15}
\definecolor{bostonuniversityred}{rgb}{0.8, 0.0, 0.0}
\definecolor{ceruleanblue}{rgb}{0.16, 0.32, 0.75}
\definecolor{airforceblue}{rgb}{0.36, 0.54, 0.66}
\definecolor{cadmiumgreen}{rgb}{0.0, 0.42, 0.24}
\definecolor{ao(english)}{rgb}{0.0, 0.5, 0.0}
\definecolor{coolblack}{rgb}{0.0, 0.18, 0.39}
\definecolor{byzantine}{rgb}{0.74, 0.2, 0.64}
\definecolor{alizarin}{rgb}{0.82, 0.1, 0.26}
\definecolor{arsenic}{rgb}{0.23, 0.27, 0.29}
\definecolor{cobalt}{rgb}{0.0, 0.28, 0.67}
\definecolor{amber}{rgb}{1.0, 0.75, 0.0}

\title{Semi-supervised Community Detection\\using Glauber Dynamics for an Ising Model \vspace{\baselineskip}}

\author{
\begin{tabular}{ccc}
\normalsize Konstantin Avrachenkov & \hspace{1cm} & \normalsize Diego Goldsztajn \\
\footnotesize Inria & \hspace{1cm} & \footnotesize Universidad ORT Uruguay \\
\footnotesize Sophia Antipolis, France & \hspace{1cm} & \footnotesize Montevideo, Uruguay \\
\scriptsize\texttt{k.avrachenkov@inria.fr} & \hspace{1cm} & \scriptsize\texttt{goldsztajn@ort.edu.uy} \\
\end{tabular}
}

\date{\vspace{\baselineskip} July 24, 2026}

\begin{document}


\maketitle

\noindent\rule{\textwidth}{1pt}

\vspace{2\baselineskip}

\onehalfspacing

\begin{adjustwidth}{0.8cm}{0.8cm}
\begin{center}
\textbf{Abstract}
\end{center}

\vspace{0.3\baselineskip}

\noindent We consider graphs with two communities and analyze an algorithm for learning the community labels when the edges of the graph and only a small fraction of the labels are known in advance. The algorithm is based on the Glauber dynamics for an Ising model where the energy function includes a quadratic penalty on the magnetization. The analysis focuses on graphs sampled from a Stochastic Block Model (SBM) with slowly growing mean degree. We derive a mean-field limit for the magnetization of each community, which can be used to choose the run-time of the algorithm to obtain a target accuracy level. We further prove that almost exact recovery is achieved in a number of iterations that is quasi-linear in the number of nodes. As a special case, our results provide the first rigorous analysis of the label propagation algorithm in the SBM with slowly diverging mean degree. We complement our theoretical results with several numerical experiments.

\vspace{\baselineskip}

\small{\noindent \textit{Key words:} semi-supervised learning, community detection, stochastic block model, Glauber dynamics, Ising model, mean-field limit, almost exact recovery.}

\vspace{0.3\baselineskip}
\small{\noindent Authors are listed alphabetically. Part of this work was done while Diego Goldsztajn was with Inria, the other part was carried out while he was with Universidad ORT Uruguay and was supported by ANII Uruguay under fellowship PD\_NAC\_2024\_182118. This paper has been accepted for publication at Bernoulli.} 
\end{adjustwidth}

\newpage

\section{Introduction}
\label{sec: introduction}

Graph clustering or community detection is a central problem in machine learning, with applications in social \cite{wasserman1994social,girvan2002community,bedi2016community} and biological \cite{bhowmick2015clustering,brohee2006evaluation} networks, bibliometrics \cite{chen2010community,vsubelj2016clustering} and image processing \cite{shi2000normalized,tolliver2006graph}, among many others. In the present paper, we study the problem of clustering the nodes of a graph into two communities when a small fraction of community labels are disclosed; this is commonly referred to as semi-supervised learning. We consider the Stochastic Block Model (SBM) with two communities, a fundamental probabilistic model where nodes in the same community are connected with higher probability than nodes in opposite communities. Also, we focus on limiting regimes where the mean degree diverges at an arbitrarily low rate and the fraction of disclosed labels remains fixed.

Our algorithm identifies the binary community labels with positive or negative spins. It is initialized by assigning the correct spin to the nodes with disclosed labels and uniformly random spins to all the other nodes. Guided by the structure of the maximum likelihood estimator for the true configuration of spins, we define an energy function on the space of these configurations. The first term of the energy function corresponds to a standard Ising model, but we also have a term that depends quadratically on the sum of the spins, penalizing the monochromatic configurations.  Our algorithm uses the continuous-time Glauber dynamics associated with the energy function to modify the spins over time. Specifically, each spin flips according to an independent Poisson process with a rate that depends on the energy reduction caused by the flip, in a way that tends to reduce the energy. We prove that this drives the configuration of spins to the true configuration, which can be regarded as a local minimum of the energy function.

\subsection{Overview of main results}
\label{sub: overview of main results}

Let $n$ be a scaling parameter, with the total number of nodes $\Theta(n)$, and suppose that the mean degree scales as $\lambda_n$. Our algorithm has two hyperparameters: $\alpha_n$ is the weight of the quadratic penalty and $\beta_n$ is the inverse temperature of the Glauber dynamics. We assume that $\lambda_n \to \infty$ and $\beta_n \lambda_n \to \infty$ as $n \to \infty$, and show that our algorithm is effective if $\alpha_n$ is in a range of values given by the connectivity parameters of the SBM, which can be estimated using the revealed labels. We obtain the following results.
\begin{enumerate}
	\item \emph{Mean-field limit.} We establish in Theorem \ref{the: mean-field limit} that the vector-valued process which describes the magnetization of each community converges weakly to the solution of a differential equation. This result can be of independent interest for the probability and statistical physics communities studying the behavior of the Ising model on random structures.
	
	\item \emph{Partial recovery.}
	Since the classification error can be expressed in terms of the magnetization vector, the differential equation can be used to set the simulation time $t_{\mathrm{end}}$ of the Glauber dynamics so that the final error is below a given threshold. Given the target error and the simulation time, our algorithm involves $O(nt_{\mathrm{end}})$ spin flips. To the best of our knowledge, we are the first to characterize the effort needed by a graph clustering algorithm to achieve a target level of accuracy.
	
	\item \emph{Almost exact recovery.} We show in Theorem \ref{the: almost exact recovery} that a vanishing fraction of incorrect community labels is obtained if the simulation time diverges slowly with $n$ instead of remaining constant. In this case, our algorithm involves $O(n \log \lambda_n)$ spin flips.
	
	\item \emph{Analysis of Label Propagation.} The above results cover the special case where $\alpha_n = 0$ and $\beta_n = \infty$ for all $n$, which corresponds to the Label Propagation or Majority Vote algorithm proposed for clustering problems in \cite{raghavan2007near}. The interdependence of nonlinear dynamics and graph topology creates significant challenges for the analysis on random graphs. To the best of our knowledge, we are the first to rigorously analyze this algorithm in the SBM with slowly diverging mean degree.
	
	\item \emph{Numerical results.} The range of admissible values for $\alpha_n$ is an interval that typically contains zero; the more symmetric the community sizes are, the larger this interval is. While our asymptotic results hold as long as $\alpha_n$ lies in this interval, we observe numerically that our algorithm performs better when the quadratic penalty is active with $\alpha_n > 0$; we explain this intuitively relying on our proof arguments. We also show numerically that our algorithm is faster than Belief Propagation \cite{decelle2011asymptotic} and performs better than other semi-supervised algorithms with similar complexity.
\end{enumerate}

Mean-field limits are unusual in the clustering literature, but as noted above, have the advantage of characterizing how the classification error evolves over time, which can be leveraged to estimate the time required by the algorithm to achieve a target accuracy level. Standard mean-field arguments can be used to describe the macroscopic behavior of particles interacting in a way that makes them statistically exchangeable; see the pioneering papers \cite{barbour1980density,kurtz1970solutions,kurtz1971limit,kurtz1978strong}. In contrast, the spins that we consider are not exchangeable because the flip rate of a spin depends on the graph neighborhood, which varies across the spins. A further challenge is that the random graph and the current configuration of spins are not independent since the flip rates of the spins depend on the graph structure. The analysis of our algorithm is even further complicated by the fact that the flip rates depend on the sums of neighboring spins in a way that becomes discontinuous in the limit as $n \to \infty$.

Our proofs use concentration inequalities to approximate the sums of neighboring spins by quantities that depend on the magnetizations of the communities and the connectivity parameters of the graph. We prove that these approximations become exact in the limit, and allow to describe the evolution of the magnetization vector through a differential equation. The almost exact recovery result does not follow from the mean-field limit, which concerns process-level convergence over finite intervals of time, and requires additional arguments for analyzing the Glauber dynamics over a diverging time horizon.

\subsection{Related work}
\label{sub: related work}

Many theoretical and practical aspects of graph clustering have been comprehensively discussed in surveys and books; see \cite{abbe2018community,schaeffer2007graph,fortunato2010community,newman2018networks,menczer2020first,avrachenkov2022statistical}. The SBM that we consider here is an inhomogeneous Erd\H{o}s-R\'enyi random graph model, a simple but fundamental model for communities in a graph that has been extensively used in the context of clustering. This model is usually analyzed in limiting regimes where the number of nodes goes to infinity, and clustering problems are classified in terms of the fraction of incorrect community labels considered admissible. In this paper we focus on partial and almost exact recovery problems. The former require that a given (and typically large) fraction of labels are correctly identified with probability tending to one, and the latter that all but a vanishing fraction of labels are correct. Other objectives are exact recovery, which involves correctly identifying all the community labels with high probability, and weak recovery, which only requires that the learned labels are positively correlated with the true labels and is most relevant in the constant-degree regime.

\subsubsection{Graph-based semi-supervised learning}
\label{subsub: semi-supervised learning}

Many papers have considered clustering problems in the unsupervised case where no labels are disclosed in advance. For this situation we refer to the foundational works \cite{abbe2015exact,mossel2015reconstruction,mossel2018proof,yun2014community}, and the survey \cite{abbe2018community}, where the information-theoretic limits and efficient algorithms for unsupervised clustering in the SBM are extensively discussed. Unlike this literature, we consider the semi-supervised setting where a small fraction of community labels are revealed in advance. This situation is prevalent in practice and can significantly improve the efficacy and efficiency of graph clustering, as observed in \cite{zhu2003ssl,zhou2003learning,chapelle2006ssl,vanengelen2020survey,song2022graph}.

In the context of the SBM, information-theoretic limits for exact recovery with side information and logarithmic mean degree are derived in \cite{saad2018community}, which further proposes a clustering algorithm based on eigenelements. Another semi-supervised algorithm, based on a constrained linear system, is defined in \cite{avrachenkov2020almost} and is shown to achieve almost exact recovery if the mean degree diverges at any rate. Both algorithms have polynomial complexity in the number of nodes; significantly larger than our quasi-linear complexity. Recently, \cite{xing2023almost} showed that almost exact recovery in the SBM is achieved by a gossiping algorithm in quasi-linear time. However, in \cite{xing2023almost} the average degree is super-logarithmic in the number of nodes, and the density of intra-cluster edges grows faster than that of inter-cluster edges, making the setup substantially less challenging than in the present paper.

A class of semi-supervised consensus and label propagation algorithms converge in quasi-linear time and have shown remarkable efficacy on both synthetic and real-world data \cite{zhu2003ssl,zhou2003learning,raghavan2007near,garza2019community}; however, their classification error has not been rigorously characterized to the best of our knowledge. Since our algorithm also has quasi-linear time complexity, our numerical study compares our algorithm against the latter class. Specifically, we consider consensus-based algorithms \cite{zhu2003ssl}, generalized Laplacian-based algorithms \cite{zhou2003learning,avrachenkov2012generalized} (including a PageRank based algorithm) and the Poisson Learning algorithm \cite{calder2020poisson}. Our numerical results show that our algorithm achieves a much smaller classification error than all these algorithms, and in addition needs to perform considerably fewer updates per node.

\subsubsection{Ising models on random graphs}
\label{subsub: ising models on random graphs}

Using Ising models and Glauber dynamics for graph clustering has been proposed and empirically evaluated for unsupervised problems in \cite{reichardt2006statistical}. In particular, it was observed that the energy function with the penalty on magnetization corresponds to the modularity of the network (as defined in \cite{newman2006modularity}) with the Erd\H{o}s-R\'enyi graph as a null model. More recently, \cite{liu2024locally} has proposed similar Glauber dynamics to achieve weak recovery in the constant-degree symmetric SBM. In \cite{liu2024locally} the information-theoretic limit is not reached and
the run-time is roughly of order $n^4$ in the number of nodes, ignoring polylogarithmic factors; the authors write that this estimate seems conservative and that a quasi-linear time in $n$ should be sufficient for weak recovery. Thus, algorithms based on Glauber dynamics for Ising models are also of interest in the contexts of unsupervised learning and weak recovery, making them widely applicable for various graph clustering tasks.

In the context of sparse, tree-like graphs, we would like to mention a series of papers \cite{kanade2016global,mossel2016local,yu2024ising} on local Belief Propagation (BP) and Ising model that address the problem of community detection in the symmetric SBM with side information. Specifically, the authors of \cite{kanade2016global} study the SBM with side information in the form of binary erasure channel and prove that in some regimes the local BP algorithm achieves the optimal expected fraction of correctly estimated labels. The authors of \cite{mossel2016local} proved the optimality of the BP algorithm in certain regimes for side information in the form of binary symmetric channel. Finally, the authors of \cite{yu2024ising} closed the gaps, proving the optimality of the BP algorithm for all tree-like SBMs and a very general form of side information. We remark that our approach works for both tree-like and denser graphs as long as the average degree diverges.

The Glauber dynamics for the standard Ising model at zero temperature correspond to the majority vote or label propagation algorithm; as noted earlier, this is the same as setting $\alpha_n = 0$ and $\beta_n = \infty$ in our model. Since the initial proposal to apply the label propagation algorithm to clustering in \cite{raghavan2007near}, many variations have been proposed and tested on synthetic and real-world data; see, e.g.,\cite{harenberg2014community}. Nevertheless, we are not aware of rigorous analytical results besides those in \cite{kothapalli2013analysis} showing that exact recovery is achieved in two rounds of the algorithm under suitable connectivity conditions in the SBM. In particular, in \cite{kothapalli2013analysis} the intra-cluster mean degree grows faster than $n^{3/4}$ and the inter-cluster mean degree grows slower than $n^{1/2}$. We remark that the regime considered in the present paper is significantly more challenging, since the intra and inter-cluster mean degrees are of the same order and approach infinity arbitrarily slowly. A sampling version of the label propagation algorithm, where updates are based on two random neighbors, has been analyzed in \cite{cruciani2019distributed} for the SBM. There, the intra-cluster mean degree must grow faster than $n^{1/2}$ and the cut is quite sparse, which makes the regime less challenging than that considered here.

The series of papers \cite{hollander2019glauber,dommers2010ising,dommers2017metastability,dommers2017Bmetastability} studies metastability of the Glauber dynamics for the standard Ising model on various random graph models. However, there are at least three important differences with the present paper. Firstly, the Ising model in the former papers is purely ferromagnetic, whereas we analyze a more complicated model with ferromagnetic as well as antiferromagnetic interactions. Secondly, \cite{hollander2019glauber,dommers2010ising,dommers2017metastability,dommers2017Bmetastability} do not study clustered models such as the SBM. Lastly, the latter papers study the asymptotics of the mean crossover time from one metastable state to another, whereas we characterize the transient trajectory towards a set of states with good properties for identifying clusters.

Finally, we observe that several papers \cite{dembo2010ising,mossel2013exact,lubetzky2017universality} have obtained conditions for rapid mixing and cutoff for the Ising model on random graphs, but these phenomena can only appear at sufficiently small inverse temperatures. Furthermore, as noted above, our results are based on the transient behavior of the Glauber dynamics prior to mixing. We feel that this different approach can present significant interest to researchers working in the areas of statistical physics and interacting particle systems.

\subsection{Standard notation}
\label{sub: standard notation}

We use the symbols $P$ and $E$ to denote the probability of events and the expectation of functions, respectively. The underlying probability measure to which these symbols refer is always clear from the context or explicitly indicated. For random variables in a common metric space $S$, we denote the weak convergence of $\set{X_n}{n \geq 1}$ to $X$ by $X_n \Rightarrow X$. If the limiting random variable $X$ is deterministic, then this is equivalent to convergence in probability. The left and right limits of a function $\map{f}{[0, \infty)}{S}$ are denoted by
\begin{equation*}
	f\left(t^-\right) \defeq \lim_{s \to t^-} f(s) \quad \text{for all} \quad t > 0 \quad \text{and} \quad f\left(t^+\right) \defeq \lim_{s \to t^+} f(s) \quad \text{for all} \quad t \geq 0,
\end{equation*}
respectively. We say that the function $f$ is c\`adl\`ag if the left limits exist for all $t > 0$ and the right limits exist for all $t \geq 0$ and satisfy that $f\left(t^+\right) = f(t)$.

\subsection{Organization of the paper}
\label{sub: organization of the paper}

The rest of the paper is organized as follows. In Section \ref{sec: problem formulation} we give precise descriptions of the semi-supervised community detection problem and our algorithm. In Section \ref{sec: main results} we state our main results: a mean-field limit for the magnetization of each community and the almost exact recovery result. In Section \ref{sec: mean-field approximation} we introduce mean-field approximations that are used in the proofs of the main results. The main results are proved in Sections \ref{sec: mean-field limit} and \ref{sec: almost exact recovery}, respectively. In Section \ref{sec: approximation errors} we prove a key result used to establish the main results. We present and discuss several numerical experiments in Section \ref{sec: simulations} and conclude the paper in Section~\ref{sec: conclusion} with open problems. Appendix \ref{app: table of notation} contains a table of notation, whereas the proofs of several auxiliary results are given in Appendices \ref{app: tightness}, \ref{app: concentration inequalities} and \ref{app: other intermediate results}.

\section{Problem formulation}
\label{sec: problem formulation}

We consider a sequence of simple and undirected graphs $\calG_n = \left(\calV_n, \calE_n\right)$. For each graph, the set of nodes $\calV_n = \calV_n^1 \cup \calV_n^2$ is split into two disjoint communities with sizes $V_n^k \defeq |\calV_n^k|$. In this paper we use the terms \emph{community} and \emph{cluster} interchangeably. We emphasize that $n$ is just a scaling parameter and $V_n \defeq |\calV_n| = V_n^1 + V_n^2$ is the total number of nodes, which scales linearly with $n$. More specifically, we assume that
\begin{equation*}
v^k \defeq \lim_{n \to \infty} \frac{V_n^k}{n} \in (0, \infty) \quad \text{for all} \quad k \in \{1, 2\}.
\end{equation*}
Each graph $\calG_n$ is sampled from a stochastic block model where the intra-community and inter-community edge probabilities are given by
\begin{equation*}
a_n \defeq \frac{a \lambda_n}{n} \quad \text{and} \quad b_n \defeq \frac{b \lambda_n}{n},
\end{equation*}
respectively, where $\lambda_n \to \infty$ as $n \to \infty$ and $a > b > 0$ are constants. Specifically, if $u \in \calV_n^k$ and $v \in \calV_n^l$ are two nodes, then the probability that they are connected by an edge is
\begin{equation*}
p_n(k, l) \defeq
\begin{cases}
a_n & \text{if} \quad k = l, \\
b_n & \text{if} \quad k \neq l.
\end{cases}
\end{equation*}

Suppose that we do not know in which community each node is except for a small fraction of the nodes. Our objective is to learn almost all the community labels from the small fraction of labels that are known in advance and the edges of the graph.

\subsection{Maximum likelihood estimator and Ising model}
\label{sub: ising model}

A configuration or labeling $\map{\sigma}{\calV_n}{\{-1, 1\}}$ is a mapping that assigns a community label to each node. We denote the space of all configurations on $\calG_n$ by $\Sigma_n$ and use the term \emph{spin}, from the interacting particle systems literature, instead of \emph{community label}. Our approach is motivated by the next standard result, which is proved in Appendix \ref{app: other intermediate results} and provides an expression for the maximum likelihood estimator of the true configuration.

\begin{proposition}
\label{prop: maximum likelihood estimator}
Let us define $p_\calG(\sigma) \defeq \condp*{\calG_n = \calG | \calV_n^1 = \sigma^{-1}\left(1\right)\ \text{and}\ \calV_n^2 = \sigma^{-1}\left(-1\right)}$ for each graph $\calG = (\calV_n, \calE)$ and configuration $\sigma \in \Sigma_n$. Then the maximum likelihood estimator can be expressed as a solution of the following optimization problem:
\begin{equation}
\label{eq: maximum likelihood estimator}
\argmax{\sigma \in \Sigma_n}\ p_\calG(\sigma) = \argmin{\sigma \in \Sigma_n} \left\{-\frac{1}{2}\sum_{u \sim v} \sigma(u)\sigma(v) + \frac{\rho_n}{2}\left[\sum_{u \in \calV_n} \sigma(u)\right]^2\right\},
\end{equation}
where $\rho_n \defeq [\log(1 - a_n) - \log(1 - b_n)] / [\log(b_n(1 - a_n)) - \log(a_n(1 - b_n))]$ and $u \sim v$ indicates that the nodes are connected by an edge in the graph $\calG$.
\end{proposition}

As above, we write $u \sim v$ if $u$ and $v$ are connected in $\calG_n$, i.e., if $\{u, v\} \in \calE_n$. Given $\alpha \in \R$, let $\alpha_n \defeq \alpha \lambda_n / n$ and define an energy (Hamiltonian) function $\map{H_n}{\Sigma_n}{\R}$ by
\begin{equation*}
H_n(\sigma) \defeq - \frac{1}{2}\sum_{u \sim v} \sigma(u)\sigma(v) + \frac{\alpha_n}{2} \left[\sum_{u \in \calV_n} \sigma(u)\right]^2 \quad \text{for all} \quad \sigma \in \Sigma_n.
\end{equation*}
This expression coincides with the objective of \eqref{eq: maximum likelihood estimator} when $\alpha_n = \rho_n$. However, we let $\alpha$ be a hyperparameter since $\rho_n$ is difficult to estimate exactly. Another hyperparameter of our algorithm is the so-called inverse temperature $\beta_n > 0$. Together with the energy function, $\beta_n$ defines a Gibbs probability measure on the space of configurations $\Sigma_n$. More specifically, the probability of configuration $\sigma$ with respect to this measure is
\begin{equation}
\label{eq: gibbs probability measure}
\left[\sum_{\sigma' \in \Sigma_n} \e^{-\beta_nH_n(\sigma')}\right]^{-1}\e^{-\beta_nH_n(\sigma)}.
\end{equation}

Our energy $H_n$ and the above Gibbs probability measure are closely related to the standard Ising model from statistical mechanics and the generalized modularity \cite{reichardt2006statistical}. In particular, the first term of $H_n$ appears in the standard energy function for the Ising model; but this energy does not include a quadratic term in the sum of the spins. The first term of $H_n$ tends to be smaller when neighboring nodes have the same spin and is minimal if all the spins have the same sign. However, it is also relatively small for the configurations where the spin is constant within each community and opposite across the communities, since the edge density is higher within the communities. If $\alpha > 0$, then the second term of $H_n$ penalizes the configurations where the sum of the spins is far from zero, and in particular the configurations where the spin is constant over the entire graph.

\subsection{Glauber dynamics}
\label{sub: glauber dynamics}

Suppose that the spin of $u \in \calV_n$ flips and let $\sigma(u)$ denote the spin of $u$ before the change occurs. As a result of this flip, the energy of configuration $\sigma \in \Sigma_n$  changes by
\begin{align*}
	&-\frac{1}{2}\sum_{\substack{v \sim w\\ v, w \neq u}} \sigma(v)\sigma(w) + \sum_{v \sim u} \sigma(u) \sigma(v) + \frac{\alpha_n}{2}\left[V_n + 2 \sum_{\substack{v < w\\ v, w \neq u}} \sigma(v) \sigma(w) - 2 \sum_{v \neq u} \sigma(u) \sigma(v)\right] \\
	&+ \frac{1}{2}\sum_{\substack{v \sim w\\ v, w \neq u}} \sigma(v)\sigma(w) + \sum_{v \sim u} \sigma(u) \sigma(v) - \frac{\alpha_n}{2}\left[V_n + 2 \sum_{\substack{v < w\\ v, w \neq u}} \sigma(v) \sigma(w) + 2 \sum_{v \neq u} \sigma(u) \sigma(v)\right].
\end{align*}
The first line is the energy after the spin flips, whereas the second line gives the opposite of the energy before the flip. Thus, the change in the energy is: 
\begin{equation}
	\label{eq: energy variation}
	\Delta_n(\sigma, u) = 2 \sigma(u) \left[\sum_{v \sim u} \sigma(v) - \alpha_n \sum_{v \neq u} \sigma(v)\right].
\end{equation}
The continuous-time Glauber dynamics are the Markov chain $\set{\bsigma_n(t)}{t \geq 0}$ on the space of configurations $\Sigma_n$ such that the spin of each node $u$ flips independently at rate
\begin{equation}
\label{eq: flip rate when beta finite}
r\left(\beta_n \Delta_n(\bsigma_n, u)\right) \quad \text{with} \quad r(x) \defeq \frac{1}{1 + \e^x} \quad \text{for all} \quad x \in \R; 
\end{equation}
throughout the paper we use boldface notation for quantities that evolve over time. In other words, a transition between two configurations $\sigma, \rho \in \Sigma_n$ may occur only if the two configurations differ exactly at one node. In that case a transition from $\sigma$ to $\rho$ happens at rate $r\left(\beta_n \left(H_n(\rho) - H_n(\sigma)\right)\right)$. It is not difficult to check that the Markov chain defined in this way is reversible and ergodic with stationary distribution \eqref{eq: gibbs probability measure}.

\begin{remark}
\label{rem: discrete-time glauber dynamics}
A discrete-time version of the Glauber dynamics can also be considered. In this version a node $u \in \calV_n$ is selected uniformly at random at each time $t$, and then the spin of this node flips with probability $r\left(\beta_n \Delta_n\left(\bsigma_n(t), u\right)\right)$. The two versions of the dynamics are closely related. Indeed, suppose that the configuration at time $t$ is given. Then the configuration reached after a spin flips has the same distribution under both versions of the dynamics. Moreover, let $\tau_n$ be the duration of a time slot in the discrete-time version. Then it is not difficult to check that the mean amount of time until a spin flips is
\begin{equation*}
\tau_n V_n\left[\sum_{u \in \calV_n} r\left(\beta_n \Delta_n \left(\bsigma_n(t), u\right)\right)\right]^{-1} \quad \text{and} \quad \left[\sum_{u \in \calV_n} r\left(\beta_n \Delta_n \left(\bsigma_n(t), u\right)\right)\right]^{-1},
\end{equation*}
for the discrete-time and the continuous-time versions, respectively.
\end{remark}

We also consider the situation when $\beta_n = \infty$ under the following convention:
\begin{equation}
\label{eq: convention for infinite inverse temperature}
r(\beta_n x) \defeq \ind{x < 0} + \frac{1}{2}\ind{x = 0} \quad \text{for all} \quad x \in \R \quad \text{if} \quad \beta_n = \infty.
\end{equation}
In this case the spin of $u \in \calV_n$ cannot flip when $\Delta_n\left(\bsigma_n, u\right) > 0$. However, the spin flips at rate $1 / 2$ if this does not change the energy and at unit rate when this reduces the energy. This corresponds to the Label Propagation or Majority Vote algorithm in \cite{raghavan2007near} when $\alpha = 0$, i.e., nodes change their spin to the value that is most popular among their neighbors.

\subsection{Semi-supervised learning algorithm}
\label{sub: semi-supervised learning algorithm}

Algorithm \ref{alg: algorithm} outlines our semi-supervised method for learning the community labels when a small fraction $\eta \in (0, 1)$ of the labels are revealed in advance. The inputs of the algorithm are the list of community labels that are disclosed in advance and the simulation time $t_{\mathrm{end}}$ for running the Glauber dynamics.

\SetKwInput{Parameter}{Parameters}
\begin{algorithm}
	
	\KwIn{disclosed community labels (for a fraction $\eta \in (0, 1)$ of the nodes) and total simulation time $t_{\mathrm{end}} \geq 0$}
	
	\Parameter{hyperparameter $\alpha_n \in \R$ and inverse temperature $\beta_n \in (0,\infty]$}
	
	\For{$u \in \calV_n$}{
		
		\uIf{\normalfont community label disclosed}{
			check community label and set $\bsigma_n(0, u) = 1$ if $u \in \calV_n^1$ or $\bsigma_n(0, u) = -1$ if $u \in \calV_n^2$;
		}\Else{
			let $\bsigma_n(0, u)$ be uniformly random in $\{-1, 1\}$;
		}
	}
	
	\While{$t \leq t_{\mathrm{end}}$}{
		simulate Glauber dynamics $\bsigma_n$ with flip rates \eqref{eq: flip rate when beta finite} if $\beta_n < \infty$ or \eqref{eq: convention for infinite inverse temperature} if $\beta_n = \infty$ (disclosed spins may change);
	}
	
	\KwOut{community labels given by $\bsigma_n(t_{\mathrm{end}})$}
	
	\vspace{3mm}
	
	\caption{Semi-supervised algorithm for learning the community labels.}
	\label{alg: algorithm}
\end{algorithm}

To simplify the exposition, we will assume that each node has probability $\eta$ of having its label disclosed, independently from all the other nodes. However, our results hold more generally: when the disclosed labels satisfy the statement of Lemma \ref{lem: initial conditions} of Appendix \ref{app: other intermediate results}. The simulation time for the Glauber dynamics can be selected as $t_{\mathrm{end}} = \log [2(1 - \eta) / \varepsilon]$ if the relative classification error should be lower than $\varepsilon$. This choice will be justified using our mean-field limit, as we explain in Section \ref{sec: main results}. Without any loss of generality, the correct labeling corresponds to nodes in community $\calV_n^1$ having positive spin and nodes in community $\calV_n^2$ having negative spin. If $\alpha = 0$ and nodes have access to the spins of neighbors, then Algorithm \ref{alg: algorithm} can be implemented in an entirely distributed manner. Otherwise, we must keep a common variable for the total magnetization.

Informally speaking, the Glauber dynamics simulated in Algorithm \ref{alg: algorithm} modify the spins of the nodes in a way that tends to reduce the energy $H_n$, and the configuration where the spins are positive in $\calV_n^1$ and negative in $\calV_n^2$ can be regarded as a local minimum of $H_n$. In Section \ref{sec: main results} we establish that the Glauber dynamics drive the initial configuration $\bsigma_n(0)$ constructed in the first steps of the algorithm towards the latter local minimum, provided that the graph is sufficiently large and $\alpha$ is selected suitably.

\section{Main results}
\label{sec: main results}

Before stating the main results, we define some notation that will be used in the rest of the paper. We denote the neighborhood of $u \in \calV_n$ and its intersection with $\calV_n^k$ by
\begin{equation*}
\calN_n(u) \defeq \set{v \in \calV_n}{v \sim u} \quad \text{and} \quad \calN_n^k(u) \defeq \set{v \in \calV_n^k}{v \sim u},
\end{equation*}
respectively, and let $N_n(u) \defeq |\calN_n(u)|$ and $N_n^k(u) \defeq |\calN_n^k(u)|$. Further, we define
\begin{equation*}
\calY_s^k(\sigma) \defeq \set{v \in \calV_n^k}{\sigma(v) = s} \quad \text{and} \quad \hat{\calY}_s^k(\sigma, u) \defeq \set{v \in \calY_s^k(\sigma)}{v \sim u}
\end{equation*}
for all $s \in \{-, +\}$ and $\sigma \in \Sigma_n$, and let $Y_s^k(\sigma) \defeq |\calY_s^k(\sigma)|$ and $\hat{Y}_s^k(\sigma, u) \defeq |\hat{\calY}_s^k(\sigma, u)|$. In addition, the (normalized) magnetization of $\calV_n^k$ under configuration $\sigma \in \Sigma_n$ is defined as
\begin{equation*}
z^k(\sigma) \defeq \frac{1}{V_n^k}\sum_{u \in \calV_n^k} \sigma(u) = \frac{Y_+^k(\sigma) - Y_-^k(\sigma)}{V_n^k} \in [-1, 1] \quad \text{and} \quad \bz_n^k(t) \defeq z^k\left(\bsigma_n(t)\right).
\end{equation*}

The vector-valued magnetization process $\bz_n = (\bz_n^1, \bz_n^2)$ takes values in $D_{\R^2}[0, \infty)$, the space of c\`adl\`ag functions from $[0, \infty)$ into $\R^2$, which we endow with the topology of uniform convergence over compact sets. The following mean-field limit is proved in Section \ref{sec: mean-field limit}.

\begin{figure}
\centering
\includegraphics{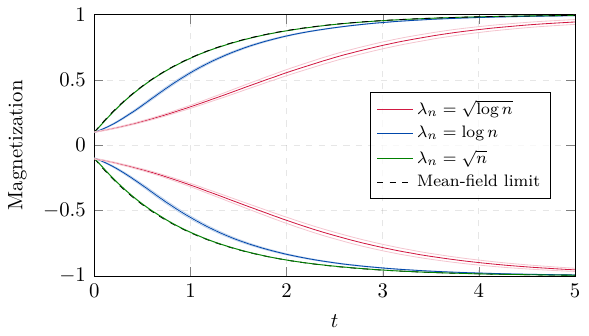}
\caption{The dashed curves show the mean-field limit, each of the dark solid curves was obtained by averaging the results of $100$ simulations and the light solid curves show $95\%$ confidence intervals around the dark solid curves; the positive curves correspond to $\bz_n^1$ and the negative curves to $\bz_n^2$. The values of $\lambda_n$ in the legend are ordered from inner curves to outer curves. In all the cases both communities have size $5000$, $a = 5$, $b = 1$, $\eta = 0.1$, $\alpha = 0$ and $\beta_n = \infty$.}
\label{fig: mean-field limit}
\end{figure}

\begin{theorem}
\label{the: mean-field limit}
Suppose that $\bsigma_n$ evolves as in Algorithm \ref{alg: algorithm} and assume that:
\begin{equation}
\label{eq: zero condition for mean-field limit}
v^1 \geq v^2 \quad \text{and} \quad \frac{bv^1 - av^2}{v^1 - v^2} < \alpha < \frac{av^1 - bv^2}{v^1 - v^2} \quad \text{if} \quad v^1 > v^2;
\end{equation}
$\alpha$ can take any real value if $v^1 = v^2$. Fix $\eta \in (0, 1)$ and a sequence $\beta_n \in (0, \infty]$ such that
\begin{equation}
\label{eq: first condition for mean-field limit}
\lim_{n \to \infty} \beta_n \lambda_n = \infty.
\end{equation}
Then $\bz_n \Rightarrow \bz_\infty$ in $D_{\R^2}[0, \infty)$ as $n \to \infty$, where $\bz_\infty$ is the unique function such that
\begin{equation*}
\bz_\infty^k(0) =
\begin{cases}
	\eta & \text{if} \quad k = 1, \\
	-\eta & \text{if} \quad k = 2,
\end{cases}
\quad \text{and} \quad
\dot{\bz}_\infty^k(t) =
\begin{cases}
	1 - \bz_\infty^k(t) & \text{if} \quad k = 1, \\
	-1 - \bz_\infty^k(t) & \text{if} \quad k = 2,
\end{cases}
\quad \text{for all} \quad t \geq 0.
\end{equation*}
Moreover, if the sequence $\gamma_n > 0$ is such that
\begin{equation}
\label{eq: second condition for mean-field limit}
\lim_{n \to \infty} \frac{\gamma_n}{\sqrt{n}} = 0, \quad \lim_{n \to \infty} \frac{\gamma_n}{\lambda_n} = 0 \quad \text{and} \quad
\lim_{n \to \infty} \gamma_n\e^{-\beta_n \lambda_n} = 0,
\end{equation}
then $\gamma_n \left(\bz_n - \bz_\infty\right) \Rightarrow 0$ in $D_{\R^2}[0, \infty)$ as $n \to \infty$.
\end{theorem}

As observed in Section \ref{sec: introduction}, the situation where $\alpha = 0$ and $\beta_n = \infty$ for all $n$ corresponds to the Label Propagation or Majority Vote algorithm, which had not been rigorously analyzed before. The plot shown in Figure \ref{fig: mean-field limit} corresponds to the latter choice of hyperparameters and illustrates how the convergence rate to the mean-field limit depends on $\lambda_n$, as is also captured by \eqref{eq: second condition for mean-field limit}. In particular, for graphs of the same finite size, the trajectories of $\bz_n$ get closer to the mean-field limit when the graph is denser.


\begin{remark}
\label{rem: selection of alpha}
If the graph is fixed, then a nonasymptotic version of \eqref{eq: zero condition for mean-field limit} is given by
\begin{equation*}
\frac{b_n V_n^1 - a_n V_n^2}{V_n^1 - V_n^2} < \alpha_n < \frac{a_n V_n^1 - b_n V_n^2}{V_n^1 - V_n^2} \quad \text{if} \quad V_n^1 > V_n^2.
\end{equation*}
While this condition depends on the possibly unknown community sizes and edge densities, these quantities can be easily estimated by considering the fraction $\eta$ of revealed labels. The community sizes can be estimated by counting the number of revealed labels in each community, while the edge densities can be estimated by counting the numbers of edges between nodes with revealed labels in the same community or in different communities.
\end{remark}

\begin{remark}
\label{rem: oracle mistakes}
Suppose that, similarly to the setting of \cite{avrachenkov2020almost}, the oracle makes mistakes for nodes in $\calV_n^k$ with probability $q^k$. Then the average magnetization when the algorithm starts is given by $\eta - 2\eta q^1$ for $\calV_n^1$ and $-\eta + 2\eta q^2$ for $\calV_n^2$. Theorems \ref{the: mean-field limit} and \ref{the: almost exact recovery}, stated below, can be extended to the situation where the oracle makes mistakes, provided that:
\begin{align*}
&\left(a - \alpha\right)v^1\left(\eta - 2\eta q^1\right) + \left(b - \alpha\right)v^2\left(-\eta + 2\eta q^2\right) > 0, \\
&\left(b - \alpha\right)v^1\left(\eta - 2\eta q^1\right) + \left(a - \alpha\right)v^2\left(-\eta + 2\eta q^2\right) < 0;
\end{align*}
observe that if $q^1 = 0 = q^2$, then we recover condition \eqref{eq: zero condition for mean-field limit}.
The proofs follow from similar arguments as in the case without mistakes, by noting that the above conditions imply that the point $\left(\eta - 2\eta q^1, -\eta + 2\eta q^2\right)$ is in the interior of the gray set depicted in Figure \ref{fig: attractivity region}.
\end{remark}

It is straightforward to check that the mean-field limit $\bz_\infty$ is such that
\begin{equation}
\label{eq: explicit expression for mean-field limit}
\bz_\infty^1(t) = 1 + \left(\eta - 1\right)\e^{-t} = - \bz_\infty^2(t) \quad \text{for all} \quad t \geq 0.
\end{equation}
As a result, $\bz_\infty^1(t) \to 1$ and $\bz_\infty^2(t) \to -1$ as $t \to \infty$, which corresponds to the spin being constant within each community and opposite between communities. If $\gamma_n$ satisfies \eqref{eq: second condition for mean-field limit} and $t$ is a fixed time, then Theorem \ref{the: mean-field limit} and \eqref{eq: explicit expression for mean-field limit} imply that
\begin{equation}
\label{eq: relative error at time t}
\max\left\{\left|\bz_n^1(t) - 1\right|, \left|\bz_n^2(t) + 1\right|\right\} = (1 - \eta) \e^{-t} + o\left(\frac{1}{\gamma_n}\right).
\end{equation}

In practice, this expression can be used to select the simulation time $t_{\mathrm{end}}$ required as input in Algorithm \ref{alg: algorithm}, e.g., in a way that keeps the relative classification error below a given threshold. If $t_{\mathrm{end}} > \log [2(1 - \eta) / \varepsilon]$, then the first term on the right-hand side is smaller than $\varepsilon / 2$ and thus the relative error is less than $\varepsilon$ if the graph is large enough. 

The expression \eqref{eq: relative error at time t} also suggests that we may achieve almost exact recovery, i.e., a vanishing fraction of misclassified nodes, if the simulation time $t_{\mathrm{end}} = t_n$ approaches infinity as $n \to \infty$. Observe that this does not follow directly from the limits stated in Theorem \ref{the: mean-field limit} since these limits hold with respect to the topology of uniform convergence over compact sets. Nonetheless, we prove in Section \ref{sec: almost exact recovery} that almost exact recovery is indeed achieved if $t_n \to \infty$ as $n \to \infty$ in the way stated in the following theorem.

\begin{theorem}
\label{the: almost exact recovery}
Assume that \eqref{eq: zero condition for mean-field limit} holds. Fix $\eta \in (0, 1)$ and a sequence $\beta_n \in (0, \infty]$ such that
\begin{equation}
\label{eq: condition for almost exact recovery}
\lim_{n \to \infty} \frac{\beta_n \lambda_n}{\log \lambda_n} = \infty.
\end{equation}
If $t_n \defeq c \log \lambda_n$ with $0 < c < 1 / 4$, and the above conditions hold, then
\begin{equation}
\label{eq: almost exact recovery}
\max\left\{\left|\bz_n^1(t_n) - 1\right|, \left|\bz_n^2(t_n) + 1\right|\right\} \Rightarrow 0 \quad \text{as} \quad n \to \infty.
\end{equation}
Moreover, the limit holds in expectation as well.
\end{theorem}

The Glauber dynamics used in Algorithm \ref{alg: algorithm} flip each spin at most at unit rate, so the aggregate flip rate is $O(n)$. It follows that the average number of flips when the dynamics are simulated for $t_n$ units of time is $O(nt_n) = O(n \log \lambda_n)$, which is quasi-linear in the size of the graph. Theorem \ref{the: almost exact recovery} says that this is enough for almost exact recovery.

\begin{remark}
\label{rem: algorithm running time}
We require that $t_n < (\log \lambda_n) / 4$ for technical reasons, but believe that almost exact recovery is also achieved if $t_n$ grows faster. However, almost exact recovery may not hold if $t_n$ grows too fast. For example, if the clusters have the same size and $\alpha = 0$, then the stationary distribution of $\bsigma_n$ given by \eqref{eq: gibbs probability measure} assigns the highest probability to the two configurations with the same spin sign for all nodes. In this case, almost exact recovery does not hold if $t_n$ is larger than the mixing time of the Glauber dynamics, which is typically large in the presence of metastable states.
\end{remark}

\section{Mean-field approximation}
\label{sec: mean-field approximation}

In this section we obtain stochastic integral equations for describing the dynamics of the vector-valued magnetization process $\bz_n$. The equations involve mean-field approximations and the associated error terms, and will be used in Section \ref{sec: mean-field limit} to prove Theorem \ref{the: mean-field limit}. As a first step for deriving these equations, we observe that, using standard arguments, we may construct the Markov chain $\bsigma_n$ such that:
\begin{align*}
\bz_n^k(t) - \bz_n^k(0) &= \frac{2}{V_n^k}\calN_+^k\left(\int_0^t \sum_{u \in \calY_-^k\left(\bsigma_n(\tau)\right)} r\left(\beta_n \Delta_n\left(\bsigma_n(\tau), u\right)\right)d\tau\right) \\
&- \frac{2}{V_n^k}\calN_-^k\left(\int_0^t \sum_{u \in \calY_+^k\left(\bsigma_n(\tau)\right)} r\left(\beta_n \Delta_n\left(\bsigma_n(\tau), u\right)\right)d\tau\right) \quad \text{for all} \quad t \geq 0,
\end{align*} 
where the processes $\calN_s^k$ are independent Poisson processes with unit rates. The sum in the first line is the rate at which a negative spin in $\calV_n^k$ flips, and the sum in the second line is the rate at which a positive spin flips. The former situation increases $\bz_n^k$ by $2 / V_n^k$, whereas the latter situation decreases $\bz_n^k$ by the same amount. Let
\begin{align*}
\calM_n^k(t) &= \calN_+^k\left(\int_0^t \sum_{u \in \calY_-^k\left(\bsigma_n(\tau)\right)} r\left(\beta_n \Delta_n\left(\bsigma_n(\tau), u\right)\right)d\tau\right) \\
&- \calN_-^k\left(\int_0^t \sum_{u \in \calY_+^k\left(\bsigma_n(\tau)\right)} r\left(\beta_n \Delta_n\left(\bsigma_n(\tau), u\right)\right)d\tau\right) \\
&+ \int_0^t \sum_{u \in \calV_n^k} \bsigma_n(\tau, u) r\left(\beta_n \Delta_n\left(\bsigma_n(\tau), u\right)\right)d\tau,
\end{align*}
which is the sum of two centered Poisson processes. Then we may write
\begin{equation}
\label{eq: stochastic equation version 1}
\bz_n^k(t) = \bz_n^k(0) + \frac{2}{V_n^k}\calM_n^k(t) - \frac{2}{V_n^k}\int_0^t \sum_{u \in \calV_n^k} \bsigma_n(\tau, u) r\left(\beta_n \Delta_n\left(\bsigma_n(\tau), u\right)\right)d\tau.
\end{equation}

The energy variation in $\sigma$ when the spin of $u \in \calV_n^k$ flips can be expressed as follows:
\begin{equation}
\label{eq: alternative expression for energy variation}
\begin{split}
\Delta_n(\sigma, u) &= 2 \sigma(u) \left[\sum_{v \sim u} \sigma(v) - \alpha_n \sum_{l = 1, 2} \sum_{v \in \calV_n^l} \sigma(v)\right] + 2\alpha_n\\
&= 2\sigma(u) \sum_{l = 1, 2} \left[\sum_{v \in \calN_n^l(u)} \sigma(v) - \alpha_n V_n^l z^l(\sigma)\right] + 2\alpha_n.
\end{split}
\end{equation}
If we let $\bark \defeq 3 - k$, then it makes sense to approximate
\begin{equation*}
\sum_{v \in \calN_n^k(u)} \bsigma_n(t, v) \quad \text{by} \quad a_n V_n^k \bz_n^k(t) \quad \text{and} \quad \sum_{v \in \calN_n^{\bark}(u)} \bsigma_n(t, v) \quad \text{by} \quad b_n V_n^{\bark} \bz_n^{\bark}(t).
\end{equation*}
The right-hand sides would be the expected values of the left-hand sides if $\bsigma_n(t)$ and $\calG_n$ were independent, but they are not independent for $t > 0$ because the Glauber dynamics depend on the graph structure. Consider the linear function $\map{\calL^k}{\R^2}{\R}$ such that
\begin{equation*}
\calL^k(z) \defeq \left(a - \alpha\right) v^k z^k + \left(b - \alpha\right) v^\bark z^\bark \quad \text{for all} \quad z \in \R^2.
\end{equation*}
If we drop the last term of \eqref{eq: alternative expression for energy variation} and we further approximate the coefficients $a_n V_n^k$, $b_n V_n^k$ and $\alpha_n V_n^k$ by $av^k \lambda_n$, $bv^k \lambda_n$ and $\alpha v^k \lambda_n$, respectively, then we may approximate
\begin{equation*}
\Delta_n(\bsigma_n(t), u) \quad \text{by} \quad 2\lambda_n\bsigma_n(t, u)\calL^k\left(\bz_n(t)\right) \quad \text{for} \quad u \in \calV_n^k.
\end{equation*}

It is not difficult to check that
\begin{equation*}
\frac{Y_+^k\left(\bsigma_n\right)}{V_n^k} = \frac{1 + \bz_n^k}{2} \quad \text{and} \quad \frac{Y_-^k\left(\bsigma_n\right)}{V_n^k} = \frac{1 - \bz_n^k}{2}.
\end{equation*}
Indeed, the sum of the quantities on the left equals one and their difference is $\bz_n^k$. Further, $r(x) + r(-x) = 1$ and $r(2x) - r(-2x) = -\tanh(x)$ for all $x \in \R$. Thus,
\begin{align*}
\frac{2}{V_n^k}\sum_{u \in \calV_n^k} \bsigma_n(u) r\left(\beta_n \Delta_n\left(\bsigma_n, u\right)\right) &= \frac{2}{V_n^k}\sum_{u \in \calY_+^k\left(\bsigma_n\right)} \left[r\left(\beta_n \Delta_n\left(\bsigma_n, u\right)\right) - r\left(2\beta_n\lambda_n\calL^k\left(\bz_n\right)\right)\right] \\
&- \frac{2}{V_n^k}\sum_{u \in \calY_-^k\left(\bsigma_n\right)}  \left[r\left(\beta_n \Delta_n\left(\bsigma_n, u\right)\right) - r\left(-2\beta_n\lambda_n\calL^k\left(\bz_n\right)\right)\right] \\
&- \tanh\left(\beta_n\lambda_n\calL^k\left(\bz_n\right)\right) + \bz_n^k.
\end{align*}

We conclude from the latter equation and \eqref{eq: stochastic equation version 1} that
\begin{equation}
\label{eq: stochastic equation version 2}
\begin{split}
\bz_n^k(t) &= \bz_n^k(0) + \int_0^t \left[\tanh\left(\beta_n\lambda_n\calL^k\left(\bz_n(\tau)\right)\right) - \bz_n^k(\tau)\right]d\tau + \frac{2}{V_n^k}\calM_n^k(t) \\
&- \frac{2}{V_n^k}\int_0^t \sum_{u \in \calV_n^k} \left[r\left(\beta_n \Delta_n\left(\bsigma_n(\tau), u\right)\right) - r\left(2\beta_n\lambda_n\bsigma_n(\tau, u)\calL^k\left(\bz_n(\tau)\right)\right)\right]d\tau.
\end{split}
\end{equation}
The first two terms are what we call the mean-field approximation. The last two terms are the errors associated with this approximation and will be shown to vanish as $n \to \infty$. It follows from \eqref{eq: convention for infinite inverse temperature} that the above equation holds with
\begin{equation*}
\tanh\left(\beta_nx\right) = r(-2\beta_n x) - r(2\beta_n x) = \sign(x) \quad \text{for all} \quad x \in \R \quad \text{if} \quad \beta_n = \infty.
\end{equation*}

For each configuration $\sigma \in \Sigma_n$, define
\begin{equation*}
\calE^k(\sigma) \defeq \frac{2}{V_n^k}\sum_{u \in \calV_n^k} \left[r\left(\beta_n \Delta_n\left(\sigma, u\right)\right) - r\left(2\beta_n\lambda_n\sigma(u)\calL^k\left(z(\sigma)\right)\right)\right].
\end{equation*}
Using this notation, we may rewrite \eqref{eq: stochastic equation version 2} as follows:
\begin{equation}
\label{eq: stochastic equation version 3}
\bz_n^k(t) = \bz_n^k(0) + \int_0^t \left[\tanh\left(\beta_n\lambda_n\calL^k\left(\bz_n(\tau)\right)\right) - \bz_n^k(\tau) - \calE^k\left(\bsigma_n(\tau)\right)\right]d\tau + \frac{2}{V_n^k}\calM_n^k(t).
\end{equation}

\section{Proof of Theorem \ref{the: mean-field limit}}
\label{sec: mean-field limit}

In this section we prove Theorem \ref{the: mean-field limit}. The first step is to prove a tightness result which implies that every convergent subsequence of $\set{\bz_n}{n \geq 1}$ has a further subsequence that converges weakly in $D_{\R_2}[0, \infty)$. By standard arguments, it then suffices to fix an arbitrary convergent subsequence $\set{\bz_m}{m \in \calM}$ and prove that $\bz_m \Rightarrow \bz_\infty$ and $\gamma_m(\bz_m - \bz_\infty) \Rightarrow 0$ in $D_{\R^2}[0, \infty)$ as $m \to \infty$ under the corresponding assumptions in Theorem \ref{the: mean-field limit}. Furthermore, it is enough to fix an arbitrary $T \geq 0$ and prove the limits for the uniform metric over the finite interval $[0, T]$. A key step for this, which is deferred until Section \ref{sec: approximation errors}, is to show that the last two terms of \eqref{eq: stochastic equation version 2} converge weakly to zero. In particular, Proposition~\ref{prop: mean-field approximation} is where we overcome the main obstacle created by the mutual dependence between the current spin configuration and the structure of the random graph. Finally, we invoke Skorohod's representation theorem to obtain versions of the processes such that the convergence holds almost surely, and prove that almost all trajectories satisfy $\bz_m \to \bz_\infty$ and $\gamma_m(\bz_m - \bz_\infty) \to 0$ uniformly over $[0, T]$ as $m \to \infty$ under the corresponding assumptions.

The tightness result is stated below and proved in Appendix \ref{app: tightness}, using the fact that the processes $\bz_n$ have jumps of size $O\left(1 / n\right)$ at rate $O(n)$.

\begin{lemma}
\label{lem: tightness of magnetization processes}
The sequence of processes $\set{\bz_n}{n \geq 1}$ is tight in $D_{\R^2}[0, \infty)$ with respect to the topology of uniform convergence over compact sets, and every subsequence of $\set{\bz_n}{n \geq 1}$ has a further subsequence that converges weakly to an almost surely continuous process.
\end{lemma}

Consider any sequence $\calM \subset \N$ such that $\bz_m \Rightarrow \bz$ in $D_{\R^2}[0, \infty)$ as $m \to \infty$, where $\bz$ is almost surely continuous and $m$ takes values in $\calM$. In order to establish Theorem \ref{the: mean-field limit}, it suffices to show that the following two properties hold for any such sequence:
\begin{enumerate}
\item[(a)] $\bz = \bz_\infty$ with probability one if \eqref{eq: zero condition for mean-field limit} and \eqref{eq: first condition for mean-field limit} hold,

\item[(b)] $\gamma_m\left(\bz_m - \bz_\infty\right) \Rightarrow 0$ in $D_{\R^2}[0, \infty)$ as $m \to \infty$ if \eqref{eq: second condition for mean-field limit} holds as well.
\end{enumerate}
Property (a) can be proved by showing that for each $T \geq 0$, we have $\bz(t) = \bz_\infty(t)$ for all $t \in [0, T]$ almost surely. Similarly, let $D_{\R^2}[0, T]$ be the space of c\`adl\`ag functions from $[0, T]$ into $\R$ with the topology of uniform convergence. Then (b) holds if $\gamma_m \left(\bz_m - \bz_\infty\right) \Rightarrow 0$ in $D_{\R^2}[0, T]$ as $m \to \infty$ for all $T \geq 0$. Thus, we fix $T \geq 0$ and prove these properties.

For each $\xi \in (0, 1)$ and $\zeta > 0$, consider the following sets:
\begin{align*}
&\calA(\zeta, \xi) \defeq \set{z \in \left[-1 + \xi, 1 - \xi\right]^2}{\min\left\{\calL^1\left(z\right), -\calL^2\left(z\right)\right\} \geq \zeta}, \\
&\Sigma_n(\zeta, \xi) \defeq \set{\sigma \in \Sigma_n}{z(\sigma) \in \calA(\zeta, \xi)}.
\end{align*}
The former set is depicted in Figure \ref{fig: attractivity region}. The value of $\tanh(\beta_n\lambda_n \calL^k\left(\bz_n\right))$ in \eqref{eq: stochastic equation version 3} is mostly given by the sign of $\calL^k\left(\bz_n\right)$, particularly when $\beta_n \lambda_n$ is large, and this sign depends on the position of $\bz_n$ relative to the line $\calL^k(z) = 0$ plotted in Figure \ref{fig: attractivity region}. We have $\calL^1(\bz_n) > 0$ and $\calL^2(\bz_n) < 0$ when $\bz_n$ is in the interior of the set $\calA(0, 0)$, and this pushes $\bz_n$ in the direction of the point $(1, -1)$, which is such that the spins are positive in $\calV_n^1$ and negative in $\calV_n^2$.

\begin{figure}
\centering
\includegraphics{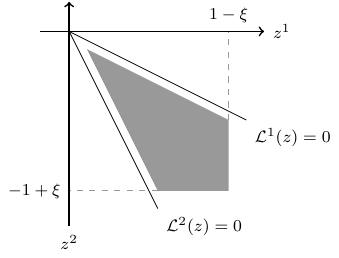}
\caption{The shaded region represents the set $\calA(\zeta, \xi)$ when condition \eqref{eq: zero condition for mean-field limit} holds and $\alpha > a$. The slopes of the lines $\calL^1(z) = 0$ and $\calL^2(z) = 0$ depend on the value of $\alpha$ as illustrated in Figure \ref{fig: sign plot} of Section \ref{sec: simulations}.}
\label{fig: attractivity region}
\end{figure}

It follows from \eqref{eq: explicit expression for mean-field limit} that
\begin{equation*}
0 < \eta \leq \bz_\infty^1(t) = -\bz_\infty^2(t) \leq 1 + (\eta - 1)\e^{-T} < 1 \quad \text{for all} \quad t \in [0, T].
\end{equation*}
We conclude from the above relations and condition \eqref{eq: zero condition for mean-field limit} that
\begin{align*}
&\calL^1\left(\bz_\infty(t)\right) = \left[(a - \alpha)v^1 - (b - \alpha)v^2\right]\bz_\infty^1(t) \geq \calL^1\left(\bz_\infty(0)\right) = \left[(a - \alpha)v^1 - (b - \alpha)v^2\right] \eta > 0, \\
&\calL^2\left(\bz_\infty(t)\right) = \left[(b - \alpha)v^1 - (a - \alpha)v^2\right]\bz_\infty^1(t) \leq \calL^2\left(\bz_\infty(0)\right) = \left[(b - \alpha)v^1 - (a - \alpha)v^2\right] \eta < 0,
\end{align*}
for all $t \geq 0$. In particular, we may write
\begin{equation}
\label{eq: lower bound l functions}
\min\left\{\calL^1\left(\bz_\infty(t)\right), -\calL^2\left(\bz_\infty(t)\right)\right\} \geq \min\left\{\calL^1\left(\bz_\infty(0)\right), -\calL^2\left(\bz_\infty(0)\right)\right\} > 0.
\end{equation}
Then there exists $\theta \in (0, 1)$ such that $\bz_\infty(t)$ is in the interior of $\calA(\theta, \theta)$ for all $t \in [0, T]$.

If \eqref{eq: zero condition for mean-field limit} and \eqref{eq: first condition for mean-field limit} hold and $\gamma_m = 1$ for all $m \in \calM$, or if \eqref{eq: zero condition for mean-field limit}, \eqref{eq: first condition for mean-field limit} and \eqref{eq: second condition for mean-field limit} hold, then:
\begin{subequations}
\label{eq: weak limits}
\begin{align}
&\bz_m \Rightarrow \bz \quad \text{in} \quad D_{\R^2}[0, T], \\
&\gamma_m\left(\frac{\calM_m^1}{V_m^1}, \frac{\calM_m^2}{V_m^2}\right) \Rightarrow 0 \quad \text{in} \quad D_{\R^2}[0, T], \\
&\gamma_m\left[\bz_m(0) - \bz_\infty(0)\right] \Rightarrow 0 \quad \text{in} \quad \R^2, \label{seq3: weak limits} \\
&\gamma_m\left(\max_{\sigma \in \Sigma_m\left(\theta, \theta\right)} \calE^1(\sigma), \max_{\sigma \in \Sigma_m\left(\theta, \theta\right)} \calE^2(\sigma)\right) \Rightarrow 0 \quad \text{in} \quad \R^2,
\end{align} 
\end{subequations}
as $m \to \infty$. The first limit holds because we are considering a convergent subsequence. The second limit is essentially a consequence of the central limit theorem for the Poisson process and is proved in Proposition~\ref{prop: limit of poisson term} of Section \ref{sec: approximation errors}. The third limit follows from the central limit theorem for independent and identically distributed random variables and is derived in Lemma \ref{lem: initial conditions} of Appendix \ref{app: other intermediate results}. The fourth limit is obtained in Proposition \ref{prop: mean-field approximation} of Section \ref{sec: approximation errors} using concentration inequalities for the binomial distribution.

\begin{lemma}
\label{lem: representation lemma}
The graphs $\calG_m$, the processes introduced in Section \ref{sec: mean-field approximation} and the process $\bz$ can be constructed on a common probability space $\left(\Omega, \calF, \prob\right)$ for all $m \in \calM$ such that the following limits hold as $m \to \infty$ with probability one:
\begin{subequations}
\label{eq: strong limits}
\begin{align}
&\bz_m \to \bz \quad \text{in} \quad D_{\R^2}[0, T], \label{seq1: strong limits} \\
&\gamma_m\left(\frac{\calM_m^1}{V_m^1}, \frac{\calM_m^2}{V_m^2}\right) \to 0 \quad \text{in} \quad D_{\R^2}[0, T], \label{seq2: strong limits} \\
&\gamma_m\left[\bz_m(0) - \bz_\infty(0)\right] \to 0 \quad \text{in} \quad \R^2, \label{seq3: strong limits} \\
&\gamma_m\left(\max_{\sigma \in \Sigma_m\left(\theta, \theta\right)} \calE^1(\sigma), \max_{\sigma \in \Sigma_m\left(\theta, \theta\right)} \calE^2(\sigma)\right) \to 0 \quad \text{in} \quad \R^2. \label{seq4: strong limits}
\end{align}  
\end{subequations}
\end{lemma}

The above technical lemma is proved in Appendix \ref{app: other intermediate results} using Skorohod's representation theorem. In order to prove Theorem \ref{the: mean-field limit}, it remains to establish that $\Omega$ contains a subset of probability one such that all $\omega$ in this subset satisfy the following properties:
\begin{enumerate}
\item[(a')] $\bz(\omega, t) = \bz_\infty(t)$ for all $t \in [0, T]$ when \eqref{eq: zero condition for mean-field limit} and \eqref{eq: first condition for mean-field limit} hold,

\item[(b')] $\gamma_m\left[\bz_m(\omega) - \bz_\infty\right] \to 0$ in $D_{\R^2}[0, T]$ as $m \to \infty$ when \eqref{eq: second condition for mean-field limit} holds as well.
\end{enumerate}
Below we establish that the above properties hold for all $\omega$ in the set of probability one where \eqref{eq: strong limits} holds and $\bz(\omega)$ is continuous.

\subsubsection*{Proof of property (a')}

Fix any $\omega$ such that \eqref{eq: strong limits} holds and $\bz(\omega)$ is continuous; in the sequel we will omit $\omega$ from the notation for brevity. In addition, consider the time defined as
\begin{equation}
\label{aux: definition of tau}
\tau \defeq \inf\set{t \in [0, T]}{\bz(t) \neq \bz_\infty(t)}.
\end{equation}
If $\tau = T$, then (a') holds, so assume that $\tau < T$. Note that $\bz$ is continuous and $\bz_m \to \bz$ uniformly over $[0, T]$ as $m \to \infty$. Hence, there exist $\varepsilon > 0$ and $m_0 \geq 1$ such that
\begin{equation}
\label{eq: magnetization is in the attractivity region}
\bz_m(t) \in \calA(\theta, \theta) \quad \text{for all} \quad t \in [0, \tau + \varepsilon] \quad \text{and} \quad m \in \calM \quad \text{with} \quad m \geq m_0. 
\end{equation}

For all $c > 0$ and $x > \theta$, we have
\begin{equation}
\label{eq: bound for hyperbolic tangent}
\left|-1 - \tanh(-c x)\right| = \left|1 - \tanh(c x)\right| = 1 - \frac{\e^{c x} - \e^{-c x}}{\e^{c x} + \e^{-c x}} \leq 2 \e^{-c \theta}.
\end{equation}
The above inequality, \eqref{eq: first condition for mean-field limit}, \eqref{seq4: strong limits} and \eqref{eq: magnetization is in the attractivity region} imply that
\begin{align*}
&\lim_{m \to \infty} \sup_{t \in [0, \tau + \varepsilon]}\left|\tanh\left(\beta_m \lambda_m \calL^k\left(\bz_m(t)\right)\right) - \ind{k = 1} + \ind{k = 2}\right| = 0, \\
&\lim_{m \to \infty} \sup_{t \in [0, \tau + \varepsilon]} \calE^k\left(\bsigma_m(t)\right) \leq \lim_{m \to \infty} \max_{\sigma \in \Sigma_m(\theta, \theta)} \calE^k(\sigma) = 0.
\end{align*}
If we replace the hyperbolic tangent in the first limit by a sign function, then the limit still holds, which covers the case $\beta_m = \infty$. Now it follows from \eqref{eq: stochastic equation version 3} and \eqref{eq: strong limits} that
\begin{equation*}
\bz^k(t) = \bz_\infty^k(0) + \int_0^t \left[\ind{k = 1} - \ind{k = 2} - \bz^k(\zeta)\right]d\zeta \quad \text{for all} \quad t \in [0, \tau + \varepsilon] \quad \text{and} \quad k \in \{1, 2\}.
\end{equation*}
Then $\bz(t) = \bz_\infty(t)$ for all $t \in [0, \tau + \varepsilon]$ contradicting \eqref{aux: definition of tau}, so (a') must hold.

\subsubsection*{Proof of property (b')}

Suppose that \eqref{eq: second condition for mean-field limit} holds and fix any $\omega$ such that \eqref{eq: strong limits} holds and $\bz(\omega)$ is continuous; we omit $\omega$ from the notation for brevity. Since (a') holds, by similar arguments as for \eqref{eq: magnetization is in the attractivity region}, there exists $m_1 \geq 1$ such that
\begin{equation*}
\bz_m(t) \in \calA(\theta, \theta) \quad \text{for all} \quad t \in [0, T] \quad \text{and} \quad m \in \calM \quad \text{with} \quad m \geq m_1. 
\end{equation*}

Then it follows from \eqref{eq: second condition for mean-field limit}, \eqref{seq4: strong limits} and \eqref{eq: bound for hyperbolic tangent} that the following limits hold:
\begin{equation}
\label{eq: some uniform limits}
\begin{split}
&\lim_{m \to \infty} \sup_{t \in [0, T]} \gamma_m\left|\tanh\left(\beta_m \lambda_m \calL^k\left(\bz_m(t)\right)\right) - \left(\ind{k = 1} - \ind{k = 2}\right)\right| = 0, \\
&\lim_{m \to \infty} \sup_{t \in [0, T]} \gamma_m\calE^k\left(\bsigma_m(t)\right) \leq \lim_{m \to \infty} \max_{\sigma \in \Sigma_m(\theta, \theta)} \gamma_m \calE^k(\sigma) = 0.
\end{split}
\end{equation}

Let $\bd_m^k(t) \defeq \gamma_m |\bz_m^k(t) - \bz_\infty^k(t)|$. The definition of $\bz_\infty$ and \eqref{eq: stochastic equation version 3} imply that
\begin{align*}
\bd_m^k(t) &\leq \bd_m^k(0) + \int_0^t \gamma_m \left|\calE^k\left(\bsigma_m(\zeta)\right)\right|d\zeta + \frac{2\gamma_m}{V_m^k}\left|\calM_m^k(t)\right| \\
&+ \int_0^t \gamma_m \left|\tanh\left(\beta_m \lambda_m \calL^k\left(\bz_m(\zeta)\right)\right) - \ind{k = 1} + \ind{k = 2}\right|d\zeta + \int_0^t \bd_m^k(\zeta)d\zeta
\end{align*}
for all $t \in [0, T]$. Taking the supremum over $s \in [0, t]$ on both sides, we obtain:
\begin{equation}
\label{eq: bound to invoke gronwall}
\begin{split}
\sup_{s \in [0, t]} \bd_m^k(s) &\leq \bd_m^k(0) + t\sup_{s \in [0, t]}\gamma_m \left|\calE^k\left(\bsigma_m(s)\right)\right| + \sup_{s \in [0, t]} \frac{2\gamma_m}{V_m^k}\left|\calM_m^k(s)\right| \\
+& t \sup_{s \in [0, t]} \gamma_m \left|\tanh\left(\beta_m \lambda_m \calL^k\left(\bz_m(s)\right)\right) - \ind{k = 1} + \ind{k = 2}\right| + \int_0^t \sup_{s \in [0, \zeta]} \bd_m^k(s)d\zeta.
\end{split}
\end{equation}

For brevity, let us write
\begin{align*}
\varepsilon_m^k &\defeq \bd_m^k(0) + T \sup_{t \in [0, T]}\gamma_m \left|\calE^k\left(\bsigma_m(t)\right)\right| + \sup_{t \in [0, T]} \frac{2\gamma_m}{V_m^k}\left|\calM_m^k(t)\right| \\
&+ T \sup_{t \in [0, T]} \gamma_m \left|\tanh\left(\beta_m \lambda_m \calL^k\left(\bz_m(t)\right)\right) - \ind{k = 1} + \ind{k = 2}\right|.
\end{align*}
We conclude from Gr\"{o}nwall's inequality, \eqref{seq2: strong limits}, \eqref{seq3: strong limits} and \eqref{eq: some uniform limits} that
\begin{equation*}
\lim_{m \to \infty} \sup_{t \in [0, T]} \bd_m^k(t) \leq \lim_{m \to \infty} \varepsilon_m^k \e^T = 0.
\end{equation*}
This proves that (b') holds and therefore completes the proof of Theorem \ref{the: mean-field limit}.

\section{Proof of Theorem \ref{the: almost exact recovery}}
\label{sec: almost exact recovery}

In this section we prove Theorem \ref{the: almost exact recovery}. By \eqref{eq: explicit expression for mean-field limit} and the triangle inequality,
\begin{equation*}
\max\left\{\left|\bz_n^1(t_n) - 1\right|, \left|\bz_n^2(t_n) + 1\right|\right\} \leq \frac{1 - \eta}{\lambda_n^c} + \max_{k = 1, 2} \bd_n^k(t_n),
\end{equation*}
where $\bd_n^k(t) \defeq |\bz_n^k(t) - \bz_\infty^k(t)|$ for all $t \geq 0$. Therefore, it suffices to show that
\begin{equation*}
\max_{k = 1, 2} \bd_n^k(t_n) \Rightarrow 0 \quad \text{as} \quad n \to \infty.
\end{equation*}
Note that $\bd_n^k(t)$ is as in Section \ref{sec: mean-field limit} with $\gamma_n = 1$. However, in Section \ref{sec: mean-field limit} the variable $t$ ranged in a fixed interval $[0, T]$ for all $n$, whereas here $t$ ranges in $[0, t_n]$ and $t_n \to \infty$ as $n \to \infty$. Hence, we must resort to different arguments to prove the above limit.

It follows from \eqref{eq: zero condition for mean-field limit} that \eqref{eq: lower bound l functions} holds for all $t \geq 0$, so we can take $\zeta > 0$ such that
\begin{equation*}
\min\left\{\calL^1\left(\bz_\infty(t)\right), -\calL^2\left(\bz_\infty(t)\right)\right\} \geq \min\left\{\calL^1\left(\bz_\infty(0)\right), -\calL^2\left(\bz_\infty(0)\right)\right\} > \zeta \quad \text{for all} \quad t \geq 0.
\end{equation*}
Let us fix some $d \in (c,1-2c)$ and define the following constant and random time:
\begin{equation*}
\xi_n \defeq \frac{1}{\lambda_n^d} \quad \text{and} \quad \tau_n \defeq \min\left\{t_n, \inf\set{t \geq 0}{\bz_n(t) \notin \calA\left(\zeta, \xi_n\right)}\right\}.
\end{equation*}

By definition of $\zeta$ and the linearity of $\calL^1$ and $\calL^2$, there exists $\theta > 0$ such that
\begin{equation*}
\max_{k = 1, 2} \bd_n^k(t) \leq \theta \quad \text{implies that} \quad \min\left\{\calL^1\left(\bz_n(t)\right), -\calL^2\left(\bz_n(t)\right)\right\} \geq \zeta.
\end{equation*}
By \eqref{eq: explicit expression for mean-field limit}, we have $\bz_n^1(t) \leq \bz_\infty^1(t) + \bd_n^1(t) = 1 + (\eta - 1)\e^{-t} + \bd_n^1(t)$. Thus, $[\xi_n + \bd_n^1(t)]\e^t \leq 1 - \eta$ gives $\bz_n^1(t) \leq 1 - \xi_n$. Using similar arguments, we obtain:
\begin{equation*}
\left[\xi_n + \max_{k = 1, 2} \bd_n^k(t)\right]\e^t \leq 1 - \eta \quad \text{implies that} \quad \bz_n(t) \in \left[-1 + \xi_n, 1 - \xi_n\right]^2.
\end{equation*}
Combining the two observations, we conclude that
\begin{equation*}
\max_{k = 1, 2} \bd_n^k(t) \leq \theta \quad \text{and} \quad \left[\xi_n + \max_{k = 1, 2} \bd_n^k(t)\right]\e^t \leq 1 - \eta \quad \text{imply that} \quad \bz_n(t) \in \calA(\zeta, \xi_n).
\end{equation*}

In order to prove \eqref{eq: almost exact recovery}, it suffices to show that
\begin{equation}
\label{eq: key step for almost exact recovery}
\left[\xi_n + \max_{k = 1, 2} \sup_{t \in [0, \tau_n]} \bd_n^k(t)\right]\e^{t_n} \Rightarrow 0 \quad \text{as} \quad n \to \infty.
\end{equation}
Indeed, $(1 - \eta)\e^{-t_n} \leq \theta$ for all sufficiently large $n$, and thus
\begin{equation*}
\left[\xi_n + \max_{k = 1, 2} \sup_{t \in [0, \tau_n]} \bd_n^k(t)\right]\e^{t_n} \leq 1 - \eta \quad \text{implies that} \quad \bz_n(t) \in \calA(\zeta, \xi_n) \quad \text{for all} \quad t \in [0, \tau_n].
\end{equation*}
Note that $\bz_n(t) \in \calA(\zeta, \xi_n)$ for all $t \in [0, \tau_n]$ implies that $\tau_n = t_n$. Hence, \eqref{eq: key step for almost exact recovery} yields
\begin{equation*}
\lim_{n \to \infty}\condp*{\left[\xi_n + \max_{k = 1, 2} \sup_{t \in [0, t_n]} \bd_n^k(t)\right]\e^{t_n} < \varepsilon} = \lim_{n \to \infty}\condp*{\left[\xi_n + \max_{k = 1, 2} \sup_{t \in [0, \tau_n]} \bd_n^k(t)\right]\e^{t_n} < \varepsilon} = 1 
\end{equation*}
if $0 < \varepsilon \leq 1 - \eta$. Thus, \eqref{eq: key step for almost exact recovery} implies \eqref{eq: almost exact recovery}, because the above limit implies that
\begin{equation*}
\max_{k = 1, 2} \bd_n^k(t_n) \leq \left[\xi_n + \max_{k = 1, 2} \sup_{t \in [0, t_n]} \bd_n^k(t)\right]\e^{t_n} \Rightarrow 0 \quad \text{as} \quad n \to \infty.
\end{equation*}

We proceed with the proof of \eqref{eq: key step for almost exact recovery}. As in \eqref{eq: bound to invoke gronwall}, we obtain:
\begin{align*}
\sup_{s \in [0, t]} \bd_n^k(s) &\leq \bd_n^k(0) + t\sup_{s \in [0, t]} \left|\calE^k\left(\bsigma_n(s)\right)\right| + \sup_{s \in [0, t]} \frac{2}{V_n^k}\left|\calM_n^k(s)\right| \\
&+ t\sup_{s \in [0, t]} \left|\tanh\left(\beta_n \lambda_n \calL^k\left(\bz_n(s)\right)\right) - \ind{k = 1} + \ind{k = 2}\right| + \int_0^t \sup_{s \in [0, \tau]} \bd_n^k(s)d\tau
\end{align*}
for all $t \geq 0$. For brevity, let us introduce the following notation:
\begin{align*}
\boldf_n^k(t) &\defeq \bd_n^k(0) + \sup_{s \in [0, t]} \frac{2}{V_n^k}\left|\calM_n^k(s)\right| \\
\bg_n^k(t) &\defeq \sup_{s \in [0, t]} \left|\calE^k\left(\bsigma_n(s)\right)\right| + \sup_{s \in [0, t]} \left|\tanh\left(\beta_n \lambda_n \calL^k\left(\bz_n(s)\right)\right) - \ind{k = 1} + \ind{k = 2}\right|.
\end{align*}
It follows from Gr\"{o}nwall's inequality and $\tau_n \leq t_n = c \log \lambda_n$ that
\begin{equation}
\label{eq: gronwall}
\sup_{t \in [0, \tau_n]} \bd_n^k(t) \leq \left[\boldf_n^k\left(\tau_n\right) + \tau_n \bg_n^k\left(\tau_n\right)\right]\e^{\tau_n} \leq \lambda_n^c\boldf_n^k\left(t_n\right) + \lambda_n^c t_n \bg_n^k\left(\tau_n\right).
\end{equation}

Next we prove that the right-hand side times $\e^{t_n} = \lambda_n^c$ vanishes. First, we have
\begin{equation}
\label{eq: limit for f}
\lambda_n^{2c} \boldf_n^k(t_n) \Rightarrow 0 \quad \text{as} \quad n \to \infty \quad \text{for all} \quad k \in \{1, 2\}
\end{equation}
by Proposition \ref{prop: limit of poisson term} of Section \ref{sec: approximation errors} and Lemma \ref{lem: initial conditions} of Appendix \ref{app: other intermediate results}; the proposition holds because $4c < 1$ yields $\left(\lambda_n^{4c} \log \lambda_n\right) / n \to 0$ as $n \to \infty$. Second, $\bsigma_n(t) \in \Sigma_n(\zeta, \xi_n)$ if $t \in [0, \tau_n]$, and
\begin{equation*}
\lim_{n \to \infty} \xi_n \lambda_n = \lim_{n \to \infty} \lambda_n^{1 - d} = \infty \quad \text{and} \quad \lim_{n \to \infty} \frac{\lambda_n^{2c}t_n}{\xi_n \lambda_n} = \lim_{n \to \infty} c \lambda_n^{2c + d - 1}\log \lambda_n = 0
\end{equation*}
since $d < 1 - 2c$. As a result, we may invoke Proposition \ref{prop: mean-field approximation} of Section \ref{sec: approximation errors} to obtain
\begin{equation*}
\lambda_n^{2c}t_n \sup_{s \in [0, \tau_n]} \left|\calE^k\left(\bsigma_n(s)\right)\right| \Rightarrow 0 \quad \text{as} \quad n \to \infty \quad \text{for all} \quad k \in \{1, 2\}.
\end{equation*}
Furthermore, $\bz_n(t) \in \calA(\zeta, \xi_n)$ for all $t \in [0, \tau_n]$, so we conclude from \eqref{eq: condition for almost exact recovery} and \eqref{eq: bound for hyperbolic tangent} that
\begin{equation}
\label{eq: limit for g}
\lambda_n^{2c} t_n \bg_n^k(\tau_n) \Rightarrow 0 \quad \text{as} \quad n \to \infty \quad \text{for all} \quad k \in \{1, 2\}.
\end{equation}

It follows from \eqref{eq: gronwall}, \eqref{eq: limit for f} and \eqref{eq: limit for g} that
\begin{equation*}
\e^{t_n} \max_{k = 1, 2} \sup_{t \in [0, \tau_n]} \bd_n^k(t) \leq \sum_{k = 1, 2} \left[\lambda_n^{2c} \boldf_n^k(t_n) + \lambda_n^{2c}t_n\bg_n^k(\tau_n)\right] \Rightarrow 0 \quad \text{as} \quad n \to \infty. 
\end{equation*}
Since $\xi_n\e^{t_n} = \lambda_n^{c - d}$ and $d > c$, we conclude that \eqref{eq: key step for almost exact recovery} holds, which proves \eqref{eq: almost exact recovery}. Note that the maximum on the left-hand side of \eqref{eq: almost exact recovery} is bounded and thus uniformly integrable over~$n$. Hence, the limit in \eqref{eq: almost exact recovery} holds in expectation as well, which proves Theorem \ref{the: almost exact recovery}.

\section{Approximation errors}
\label{sec: approximation errors}

In this section we establish some of the intermediate results that we used in the proofs of Theorems \ref{the: mean-field limit} and \ref{the: almost exact recovery}. In particular, we bound the last two terms on the right-hand side of~\eqref{eq: stochastic equation version 2}. The following proposition takes care of the first of these error terms.

\begin{proposition}
\label{prop: limit of poisson term}
Consider sequences $\gamma_n, t_n > 0$ such that $\gamma_n^2 t_n / n \to 0$ as $n \to \infty$. Then
\begin{equation}
\label{eq: bound for centered poisson process}
\sup_{t \in [0, t_n]} \frac{\gamma_n}{V_n^k}\left|\calM_n^k(t)\right| \Rightarrow 0 \quad \text{as} \quad n \to \infty \quad \text{for all} \quad k \in \{1, 2\}.
\end{equation}
If $\gamma_n / \sqrt{n} \to 0$ as $n \to \infty$, then we further have
\begin{equation}
\label{eq: weak limit for centered poisson processes}
\gamma_n\left(\frac{\calM_n^1}{V_n^1}, \frac{\calM_n^2}{V_n^2}\right) \Rightarrow 0 \quad \text{in} \quad D_{\R^2}[0, \infty) \quad \text{as} \quad n \to \infty.
\end{equation}
\end{proposition}

\begin{proof}
Let $\calN$ be a Poisson process with unit rate. By Doob's submartingale inequality,
\begin{equation}
\label{eq: doob submartingale inequality}
\condp*{\sup_{t \in [0, t_n]} \gamma_n \left|\frac{\calN(nt)}{n} - t\right| \geq \varepsilon} \leq \left(\frac{\gamma_n}{\varepsilon n}\right)^2 \conde*{\left|\calN(n t_n) - nt_n\right|^2} = \left(\frac{\gamma_n}{\varepsilon}\right)^2 \frac{t_n}{n}.
\end{equation}

Consider the processes defined as
\begin{equation*}
\bi_n^{k, s}(t) \defeq \int_0^t \sum_{u \in \calY_s^k(\bsigma_n(\tau))} r\left(\beta_n \Delta_n\left(\bsigma_n(\tau), u\right)\right)d\tau \leq V_n^k t \quad \text{for all} \quad t \geq 0.
\end{equation*}
Since $V_n^k / n \to v^k$ as $n \to \infty$, there exists $c^k > 0$ such that $V_n^k \leq c^k n$ for all $n \geq 1$. Thus,
\begin{align*}
\sup_{t \in [0, t_n]} \frac{\gamma_n}{V_n^k} \left|\calM_n^k(t)\right| &\leq \sup_{t \in [0, t_n]} \frac{\gamma_n}{V_n^k} \left|\calN_+^k\left(\bi_n^{k, -}(t)\right) - \bi_n^{k, -}(t)\right| + \sup_{t \in [0, t_n]} \frac{\gamma_n}{V_n^k} \left|\calN_-^k\left(\bi_n^{k, +}(t)\right) - \bi_n^{k, +}(t)\right| \\
&\leq \sup_{t \in [0, c^k t_n]} \frac{n\gamma_n}{V_n^k}\left|\frac{\calN_+^k(nt)}{n} - t\right| + \sup_{t \in [0, c^k t_n]} \frac{n\gamma_n}{V_n^k}\left|\frac{\calN_-^k(nt)}{n} - t\right|.
\end{align*}
It follows from \eqref{eq: doob submartingale inequality} that \eqref{eq: bound for centered poisson process} holds.

If $\gamma_n / \sqrt{n} \to 0$ as $n \to \infty$, then \eqref{eq: bound for centered poisson process} implies that
\begin{equation*}
\sup_{t \in [0, T]} \frac{\gamma_n}{V_n^k}\left|\calM_n^k(t)\right| \Rightarrow 0 \quad \text{as} \quad n \to \infty \quad \text{for all} \quad T \geq 0 \quad \text{and} \quad k \in \{1, 2\}.
\end{equation*}
We obtain \eqref{eq: weak limit for centered poisson processes} as a straightforward consequence of this observation.
\end{proof}

For each configuration $\sigma \in \Sigma_n$, recall that
\begin{equation*}
\calE^k(\sigma) \defeq \frac{2}{V_n^k}\sum_{u \in \calV_n^k} \left[r\left(\beta_n \Delta_n\left(\sigma, u\right)\right) - r\left(2\beta_n\lambda_n\sigma(u)\calL^k\left(z(\sigma)\right)\right)\right].
\end{equation*}
Given two constants $\xi \in (0, 1)$ and $\zeta > 0$, we also recall that:
\begin{align*}
&\calA(\zeta, \xi) \defeq \set{z \in \left[-1 + \xi, 1 - \xi\right]^2}{\min\left\{\calL^1\left(z\right), -\calL^2\left(z\right)\right\} \geq \zeta}, \\
&\Sigma_n(\zeta, \xi) \defeq \set{\sigma \in \Sigma_n}{z(\sigma) \in \calA(\zeta, \xi)}.
\end{align*}
The next proposition is used in the proofs of Theorems \ref{the: mean-field limit} and \ref{the: almost exact recovery} to bound the last term on the right-hand side of \eqref{eq: stochastic equation version 2}. Its proof uses two lemmas that we derive in Appendix \ref{app: concentration inequalities}.

\begin{proposition}
\label{prop: mean-field approximation}
Consider sequences $\gamma_n > 0$ and $\xi_n \in (0, 1)$ such that
\begin{equation}
\label{eq: assumptions for mean-field approximation}
\lim_{n \to \infty} \xi_n \lambda_n = \infty \quad \text{and} \quad \lim_{n \to \infty} \frac{\gamma_n}{\xi_n \lambda_n} = 0.
\end{equation}
Then the following limits hold:
\begin{equation*}
\max_{\sigma \in \Sigma_n(\zeta, \xi_n)} \gamma_n \calE^k(\sigma) \Rightarrow 0 \quad \text{as} \quad n \to \infty \quad \text{for all} \quad \zeta > 0 \quad \text{and} \quad k \in \{1, 2\}.
\end{equation*}
\end{proposition}

\begin{proof}
Recall the notation $\bark \defeq 3 - k$. It follows from \eqref{eq: alternative expression for energy variation} and $|z_k| \leq 1$ that
\begin{equation}
\label{eq: distance between delta and l}
\begin{split}
\left|\Delta_n(\sigma, u) - 2\lambda_n\sigma(u)\calL^k\left(z(\sigma)\right)\right| &\leq 2\left|\sum_{v \in \calN_n(u)} \sigma(v) - \lambda_n\left[a v^k z^k(\sigma) + b v^\bark z^\bark(\sigma)\right]\right| \\ 
&+ 2\alpha \lambda_n \left(\left|\frac{V_n^k}{n} - v^k\right| + \left|\frac{V_n^\bark}{n} - v^\bark\right|\right) + 2\alpha_n.
\end{split}
\end{equation}
Below we show that the right-hand side is small for most nodes $u \in \calV_n^k$ when $\sigma \in \Sigma_n(\zeta, \xi_n)$ and $n$ is sufficiently large. We will conclude that $\Delta_n(\sigma, u)$ and $2\lambda_n\sigma(u)\calL^k\left(z(\sigma)\right)$ have the same sign and then prove the proposition by leveraging the following inequality:
\begin{equation}
\label{eq: bound for variation in r}
\left|r(\beta_n x) - r(\beta_n y)\right| \leq \e^{-\beta_n\min\{|x|, |y|\}} \quad \text{if} \quad \sign(x) = \sign(y).
\end{equation}
This inequality holds with zero on the right-hand side if $\beta_n = \infty$ and we use \eqref{eq: convention for infinite inverse temperature}.

Fix $\zeta > 0$ and $k \in \{1, 2\}$. Then choose some $0 < 2c < \zeta$, $d > 0$ and $n_0 \geq 1$ such that
\begin{equation*}
\left|\frac{V_n^k}{n} - v^k\right| + \left|\frac{V_n^\bark}{n} - v^\bark\right| \leq \frac{c}{\max\{\alpha, a\}} \quad \text{and} \quad 3 \sqrt{\frac{a V_n^k}{n}} + 2 \sqrt{\frac{b V_n^\bark}{n}} \leq d \quad \text{for all} \quad n \geq n_0.
\end{equation*}
In addition, choose constants $\delta_n > 0$ satisfying the following conditions:
\begin{equation}
\label{eq: conditions on gamma and delta}
\lim_{n \to \infty} \delta_n = \infty, \quad \lim_{n \to \infty} \frac{\delta_n}{\xi_n \lambda_n} = 0 \quad \text{and} \quad \lim_{n \to \infty} \frac{\delta_n}{\gamma_n} = \infty.
\end{equation}
If $\sqrt{\xi_n\lambda_n \gamma_n} \to \infty$ as $n \to \infty$, then take $\delta_n = \sqrt{\xi_n\lambda_n \gamma_n}$. Otherwise, $\gamma_n \to 0$ as $n \to \infty$, possibly only along some subsequence. If $\gamma_n \to 0$ as $n \to \infty$, then the third condition is implied by the first one and we can easily choose a sequence $\delta_n$ that satisfies the first two conditions. Now for each $\sigma \in \Sigma_n$ and $u \in \calV_n^k$, define
\begin{align*}
&A_n^k(u) \defeq \left\{\left|N_n^k(u) - a_n V_n^k\right| \geq \sqrt{a_n V_n^k \delta_n}\right\}, \\
&B_s^{k, l}(\sigma, u) \defeq \left\{\left|\hat{Y}_s^l(\sigma, u) - p_n(k, l) Y_s^l(\sigma)\right| \geq \sqrt{p_n(k, l) Y_s^l(\sigma) \delta_n}\right\}.
\end{align*}
Moreover, consider the event defined as
\begin{equation*}
C_s^k(\sigma, u) \defeq A_n^k(u) \bigcup \left[\bigcup_{(l, t) \neq (k, s)} B_t^{k, l}(\sigma, u)\right] \quad \text{for all} \quad \sigma \in \Sigma_n \quad \text{and} \quad u \in \calY_s^k(\sigma).
\end{equation*}

Lemmas \ref{lem: concentration of intracommunity degrees} and \ref{lem: concentration of signed degrees} of Appendix \ref{app: concentration inequalities} bound the probabilities of $A_n^k(u)$ and $B_t^{k, l}(\sigma, u)$. In the latter case, the bound only holds when $(l, t) \neq (k, s)$ since it relies on the sets $\hat{Y}_t^l(\sigma, u)$ being independent over $u \in \calY_s^k(\sigma)$. However, we can control the sets $\hat{Y}_s^k(\sigma, u)$ by considering the event $A_n^k(u)$ in combination with the events $B_t^{k, l}(\sigma, u)$, as done below.

Suppose that $\sigma \in \Sigma_n$ and $u \in \calY_s^k(\sigma)$. If we let $\bars \defeq -s$, then
\begin{align*}
\sum_{v \in \calN_n(u)} \sigma(v) &= \hat{Y}_+^k(\sigma, u) - \hat{Y}_-^k(\sigma, u) + \hat{Y}_+^\bark(\sigma, u) - \hat{Y}_-^\bark(\sigma, u) \\
&= s\left[\hat{Y}_+^k(\sigma, u) + \hat{Y}_-^k(\sigma, u)\right] - 2s\hat{Y}_\bars^k(\sigma, u) + \hat{Y}_+^\bark(\sigma, u) - \hat{Y}_-^\bark(\sigma, u) \\
&= sN_n^k(u) - 2s\hat{Y}_\bars^k(\sigma, u) + \hat{Y}_+^\bark(\sigma, u) - \hat{Y}_-^\bark(\sigma, u).
\end{align*}
Similarly, we have $Y_+^k(\sigma) - Y_-^k(\sigma) = sV_n^k - 2sY_\bars^k(\sigma)$. Therefore, it follows from the above identity that $n \geq n_0$ and $\left[C_s^k(\sigma, u)\right]^c$ imply that
\begin{align*}
\left|\sum_{v \in \calN_n(u)} \sigma(v) - \sum_{l = 1, 2} p_n(k, l)\left[Y_+^l(\sigma) - Y_-^l(\sigma)\right]\right| &\leq \left|N_n^k(u) - a_nV_n^k\right| + 2\left|\hat{Y}_\bars^k(\sigma, u) - a_nY_\bars^k(\sigma)\right| \\
&+ \left|\hat{Y}_+^\bark(\sigma, u) - b_nY_+^\bark(\sigma)\right| + \left|\hat{Y}_-^\bark(\sigma, u) - b_nY_-^\bark(\sigma)\right| \\
&\leq d\sqrt{\lambda_n\delta_n}, 
\end{align*}
where the last inequality follows from the definition of $C_s^k(\sigma, u)$ and the observation that $Y_t^l(\sigma) \leq V_n^l$. Also, it is easy to check that the second term on the left-hand side is at most at distance $c\lambda_n$ from $\lambda_n\left[a v^k z^k(\sigma) + b v^\bark z^\bark(\sigma)\right]$. Thus, $n \geq n_0$ and $\left[C_s^k(\sigma, u)\right]^c$ imply that
\begin{equation*}
\left|\sum_{v \in \calN_n(u)} \sigma(v) - \lambda_n\left[a v^k z^k(\sigma) + b v^\bark z^\bark(\sigma)\right]\right| \leq c\lambda_n + d\sqrt{\lambda_n \delta_n}.
\end{equation*}

Let us fix $n_1 \geq n_0$ such that
\begin{equation*}
\psi_n \defeq 2c\lambda_n + d\sqrt{\lambda_n \delta_n} + \alpha_n < \zeta \lambda_n \quad \text{for all} \quad n \geq n_1.
\end{equation*}
If $\sigma \in \Sigma_n(\zeta, \xi_n)$, then $|\calL^k(z(\sigma))| \geq \zeta$. Hence, by \eqref{eq: distance between delta and l}, $\Delta_n(\sigma, u)$ and $2\lambda_n \sigma(u)\calL^k\left(z(\sigma)\right)$ have the same sign if $C_s^k(\sigma, u)$ does not hold, $u \in \calY_s^k(\sigma)$ and $n \geq n_1$. Then \eqref{eq: bound for variation in r} yields
\begin{equation*}
\left|r\left(\beta_n\Delta_n(\sigma, u)\right) - r\left(2\beta_n\lambda_n\sigma(u)\calL^k\left(z(\sigma)\right)\right)\right| \leq \e^{-2\beta_n \left(\zeta\lambda_n - \psi_n\right)} + \ind{C_s^k(\sigma, u)}
\end{equation*}
for all $\sigma \in \Sigma_n(\zeta, \xi_n)$, $u \in \calY_s^k(\sigma)$ and $n \geq n_1$. Thus, $\sigma \in \Sigma_n(\zeta, \xi_n)$ and $n \geq n_1$ imply that
\begin{align*}
\gamma_n \calE^k(\sigma) &\leq 2\gamma_n \e^{-2\beta_n \left(\zeta\lambda_n - \psi_n\right)} + \frac{2\gamma_n}{V_n^k} \sum_{s = -1, 1} \sum_{u \in \calY_s^k(\sigma)} \ind{C_s^k(\sigma, u)} \\
&\leq 2\gamma_n \e^{-2\beta_n \left(\zeta\lambda_n - \psi_n\right)} + \frac{2\gamma_n}{V_n^k} \sum_{u \in \calV_n^k} \ind{A_n^k(u)} + \frac{2\gamma_n}{V_n^k} \sum_{s = -1, 1} \sum_{(l, t) \neq (k, s)} \sum_{u \in \calY_s^k(\sigma)} \ind{B_t^{k, l}(\sigma, u)}.
\end{align*}
It follows from $2c < \zeta$ and \eqref{eq: conditions on gamma and delta} that the first term on the right vanishes as $n \to \infty$. The proof is completed by invoking Lemmas \ref{lem: concentration of intracommunity degrees} and \ref{lem: concentration of signed degrees} from Appendix \ref{app: concentration inequalities}, which hold by \eqref{eq: conditions on gamma and delta} and imply that the last two terms on the right-hand side converge weakly to zero as $n \to \infty$.
\end{proof}

\section{Numerical experiments}
\label{sec: simulations}

In this section we discuss several numerical experiments.\footnote{Code available at \url{https://github.com/diegogolds/ising_clustering}.} First, we evaluate how the hyperparameters $\alpha$ and $\beta_n$ may affect the classification error of Algorithm \ref{alg: algorithm}. Next, we compare Algorithm \ref{alg: algorithm} with Belief Propagation. Finally, we compare our algorithm with several other semi-supervised community detection algorithms.

\subsection{Parameter selection}
\label{sub: parameter selection}

Table \ref{tab: parameter evaluation} shows the mean relative classification error (the fraction of correct spins) of our algorithm in logarithmic-degree regimes for different values of $\alpha$ and $\beta_n$; for the experiments here and below the parameters of the graph are given in the table captions. The relative errors were roughly $0\%$, $42.9\%$, $57.1\%$, and $100\%$ in a few experiments; the second and third values correspond to the monochromatic outcomes while the last value corresponds to labeling each node with the opposite community. The averages shown in Table \ref{tab: parameter evaluation} depend on how many times each of the above relative errors was observed over all the experiments.

We observe that the algorithm performs much better for $\alpha = 6$ than $\alpha = 0$. In particular, the mean relative error with $\alpha = 6$ is negligible for all $\eta \geq 0.02$, whereas the algorithm struggles with small values of $\eta$ when $\alpha = 0$, especially in the more challenging regime with $a = 7$. The value of $\beta_n$ does not affect the mean relative error significantly. However, it makes the algorithm stop in fewer iterations because eventually the flip rate of every spin is equal to zero. The average number of iterations observed when $\beta_n = \infty$ ranges roughly between $8.3 \times 10^3$ and $1.2 \times 10^4$ 
iterations across all regimes considered, which is significantly less than the $5 \times 10^4$ iterations in the cases where $\beta_n = 1$.

\renewcommand{\tabcolsep}{3.5pt}
\begin{table}
	\centering
	\renewcommand{\arraystretch}{0.7}
	{\footnotesize
		\begin{tabular}{c *{16}{c}}
			\toprule
			& \multicolumn{8}{c}{$a = 7$ and $b = 1$} & \multicolumn{8}{c}{$a = 10$ and $b = 1$} \\
			\cmidrule(lr){2-9} \cmidrule(l){10-17}
			& \multicolumn{4}{c}{$\alpha = 0$} & \multicolumn{4}{c}{$\alpha = 6$} & \multicolumn{4}{c}{$\alpha = 0$} & \multicolumn{4}{c}{$\alpha = 6$} \\
			\cmidrule(lr){2-5} \cmidrule(lr){6-9} \cmidrule(lr){10-13} \cmidrule(l){14-17}
			& \multicolumn{2}{c}{$\beta_n = 1$} & \multicolumn{2}{c}{$\beta_n = \infty$} & \multicolumn{2}{c}{$\beta_n = 1$} & \multicolumn{2}{c}{$\beta_n = \infty$} & \multicolumn{2}{c}{$\beta_n = 1$} & \multicolumn{2}{c}{$\beta_n = \infty$} & \multicolumn{2}{c}{$\beta_n = 1$} & \multicolumn{2}{c}{$\beta_n = \infty$} \\
			\cmidrule(lr){2-3} \cmidrule(lr){4-5} \cmidrule(lr){6-7} \cmidrule(lr){8-9} \cmidrule(lr){10-11} \cmidrule(lr){12-13} \cmidrule(lr){14-15} \cmidrule(l){16-17}
			$\eta$ & $\mu$ & $\sigma$ & $\mu$ & $\sigma$ & $\mu$ & $\sigma$ & $\mu$ & $\sigma$ & $\mu$ & $\sigma$ & $\mu$ & $\sigma$ & $\mu$ & $\sigma$ & $\mu$ & $\sigma$ \\
			\midrule
			0.01 & 31.4 & 21.0 & 32.9 & 22.2 & 10.0 & 30.0 & 10.0 & 30.0 & 21.4 & 26.5 & 21.1 & 22.5 & 30.0 & 45.8 & 10.0 & 30.0 \\
			0.02 & 25.7 & 21.0 & 37.1 & 19.4 & 0.00 & 0.00 & 0.00 & 0.00 & 8.57 & 17.1 & 12.9 & 19.6 & 0.00 & 0.00 & 0.00 & 0.00 \\
			0.03 & 25.7 & 21.0 & 25.7 & 21.0 & 0.00 & 0.00 & 0.00 & 0.00 & 4.29 & 12.9 & 8.57 & 17.1 & 0.00 & 0.00 & 0.00 & 0.00 \\
			0.04 & 18.6 & 23.1 & 17.1 & 21.0 & 0.00 & 0.00 & 0.00 & 0.00 & 0.00 & 0.00 & 4.29 & 12.9 & 0.00 & 0.00 & 0.00 & 0.00 \\
			0.05 & 17.1 & 21.0 & 17.1 & 21.0 & 0.00 & 0.00 & 0.00 & 0.00 & 0.00 & 0.00 & 0.00 & 0.00 & 0.00 & 0.00 & 0.00 & 0.00 \\
			0.06 & 4.29 & 12.9 & 8.57 & 17.1 & 0.00 & 0.00 & 0.00 & 0.00 & 0.00 & 0.00 & 0.00 & 0.00 & 0.00 & 0.00 & 0.00 & 0.00 \\
			0.07 & 0.00 & 0.00 & 0.00 & 0.00 & 0.00 & 0.00 & 0.00 & 0.00 & 0.00 & 0.00 & 0.00 & 0.00 & 0.00 & 0.00 & 0.00 & 0.00 \\
			\bottomrule
		\end{tabular}
	}
	\caption{Percentage classification error when $V_n^1 = 10000$, $V_n^2 = 7500$ and $n = 10000$. The headers $\mu$ and $\sigma$ refer to the mean and standard deviation over $10$ experiments, respectively. In all the experiments Algorithm \ref{alg: algorithm} is run for at most $5 \times 10^4$ iterations in total; when $\beta_n = \infty$ the algorithm may stop sooner.}
	\label{tab: parameter evaluation}
\end{table}

The effect of the hyperparameter $\alpha$ on performance can be understood intuitively by considering the mean-field approximation in \eqref{eq: stochastic equation version 3}. The dynamics of $\bz_n$ are mostly governed by the terms $\tanh(\beta_n \lambda_n\calL^k(\bz_n))$, which can be approximated by $\sign(\calL^k(\bz_n))$ when $\beta_n \lambda_n$ is large. The latter terms drive $\bz_n$ in the direction indicated by the diagrams of Figure \ref{fig: sign plot}. If $\alpha$ satisfies \eqref{eq: zero condition for mean-field limit}, then this direction points to $(1, -1)$ in a half-cone that contains $(\eta, -\eta)$, but $\bz_n$ is driven towards $(1, 1)$ or $(-1, -1)$ in the adjacent regions if $\alpha  < (a + b) / 2$. The random error terms in \eqref{eq: stochastic equation version 3} can push $\bz_n$ outside the half-cone where the mean-field drift points in the right direction, particularly if $\eta$ is small. Nevertheless, the mean-field drift points into the half-cone near its boundary when $\alpha > (a + b) / 2$. Therefore, setting $\alpha > (a + b) / 2$ makes the algorithm more robust against the random error terms in \eqref{eq: stochastic equation version 3} and results in small classification errors for a broader range of values of $\eta$. Observe that, in practice, $\alpha$ can be selected so that $\alpha > (a + b) / 2$ using Remark \ref{rem: selection of alpha}.

\begin{figure}
\centering
\begin{subfigure}{0.24\columnwidth}
\centering
\includegraphics{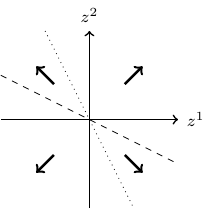}
\caption{$\alpha < b$}
\end{subfigure}
\hfill
\begin{subfigure}{0.24\columnwidth}
\centering
\includegraphics{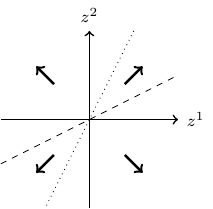}
\caption{$b < \alpha < (a + b) / 2$}
\end{subfigure}
\hfill
\begin{subfigure}{0.24\columnwidth}
\centering
\includegraphics{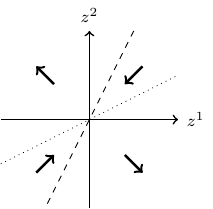}
\caption{$(a + b) / 2 < \alpha < a$}
\end{subfigure}
\hfill
\begin{subfigure}{0.24\columnwidth}
\centering
\includegraphics{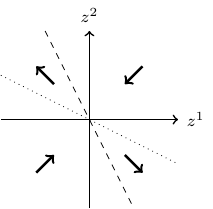}
\caption{$a < \alpha$}
\end{subfigure}
\caption{The arrows show the direction of the vector $(\sign(\calL^1(z)), \sign(\calL^2(z)))$ when \eqref{eq: zero condition for mean-field limit} holds. The dotted and dashed lines are $\calL^1(z) = 0$ and $\calL^2(z) = 0$, respectively.}
\label{fig: sign plot}
\end{figure}

\subsection{Comparison with Belief Propagation}
\label{sub: comparison with belief propagation}

The Belief Propagation (BP) algorithm is one of the main clustering algorithms for the SBM; see \cite{decelle2011asymptotic,mossel2015reconstruction,mossel2018proof}. It has been established that this algorithm achieves almost exact recovery with a complexity that is quasi-linear in the number of nodes. Even though the complexity of Algorithm \ref{alg: algorithm} is of the same order of magnitude, we show in this section that Algorithm \ref{alg: algorithm} can be substantially faster in reaching the same level of accuracy.

In order to compare both algorithms fairly, we considered the semi-supervised version of BP studied in~\cite{zhang2014phase}. Specifically, we ran the \textsc{BP-inference} algorithm described in \cite{decelle2011asymptotic} providing the exact parameters of the SBM as inputs and fixing the messages sent by nodes whose labels have been disclosed. This semi-supervised version of BP is significantly faster than the unsupervised version, which must run \textsc{BP-inference} multiple times in order to estimate the SBM parameters; see \textsc{BP-learning} in \cite{decelle2011asymptotic}. In addition, note that estimating these parameters introduces additional error.

For the comparison with BP, we considered the discrete-time version of Algorithm \ref{alg: algorithm} described in Remark \ref{rem: discrete-time glauber dynamics}, with infinite inverse temperature. To execute this algorithm, we must keep track of the spin of each node, the total magnetization and the magnetization of each neighborhood. Each iteration of the algorithm selects a node uniformly at random and updates its spin if this reduces the energy. In that case, the total magnetization and the magnetization of neighborhoods containing the spin that flipped must be updated by adding or subtracting $2$, depending on the sign of the flip. On the other hand, the semi-supervised version of BP keeps track of one variable for each directed edge and two variables for each node; the former variables are the messages that the nodes exchange. Each iteration of BP updates all variables for nodes whose labels were not disclosed.

The updates carried out by BP are considerably more costly than those performed by Algorithm~\ref{alg: algorithm}, as they involve products of sums of messages over neighborhoods; see \cite[Equations (26)-(28)]{decelle2011asymptotic}. Therefore, the mean number of updates that each algorithm performs until it reaches a given target accuracy is a complexity metric that favors BP. Nonetheless, we observe in Table \ref{tab: comparison with belief propagation} that Algorithm \ref{alg: algorithm} is roughly between $5$ and $7$ times faster with respect to this metric. Note that this metric has the advantage of not depending on the computer code implementation, which can affect the running time in seconds.

\renewcommand{\tabcolsep}{6pt}
\begin{table}
	\centering
	\renewcommand{\arraystretch}{0.7}
	{\footnotesize
		\begin{tabular}{c ccccc ccccc}
			\toprule
			& \multicolumn{5}{c}{$n = 50\ 000$ / Target error (\%)} & \multicolumn{5}{c}{$n = 500\ 000$ / Target error (\%)} \\
			\cmidrule(l){2-6}\cmidrule(l){7-11}
			$\eta$ & 0.4 & 0.8 & 1.2 & 1.6 & 2.0 & 0.4 & 0.8 & 1.2 & 1.6 & 2.0 \\
			\cmidrule(l){1-1}\cmidrule(l){2-6}\cmidrule(l){7-11}
			0.02 & 5.82 & 5.75 & 5.83 & 5.96 & 5.99 & 6.27 & 6.38 & 6.45 & 6.56 & 6.59 \\
			0.04 & 6.06 & 5.29 & 5.27 & 5.39 & 5.38 & 5.53 & 5.66 & 5.72 & 5.81 & 5.87\\
			0.06 & 5.21 & 5.37 & 5.45 & 5.52 & 5.59 & 5.71 & 5.83 & 5.92 & 5.99 & 6.06\\
			\bottomrule
		\end{tabular}
	}
	\caption{Ratio $\mathrm{\textsc{ScoreBP}} / \mathrm{\textsc{ScoreIS}}$ when $V_n^1 = V_n^2 = n$, $\lambda_n = \log n$, $a = 3$ and $b = 1$. For each value of $\eta$ and target error, the scores are averaged over $10$ experiments. Algorithm \ref{alg: algorithm} uses $\alpha = 10$ and $\beta_n = \infty$.}
	\label{tab: comparison with belief propagation}
\end{table}

More precisely, Table \ref{tab: comparison with belief propagation} considers the following metric or score for Algorithm \ref{alg: algorithm}:
\begin{equation}
	\label{eq: score ising}
	\mathrm{\textsc{ScoreIS}} = \mathrm{\textsc{NoFlipIS}} + \left[3 + (a + b) \lambda_n\right] \mathrm{\textsc{FlipIS}},
\end{equation}
where \textsc{FlipIS} and \textsc{NoFlipIS} are the numbers of iterations with and without a spin flip, respectively. In each iteration we must compute the energy variation for flipping the selected node and check if it is negative; this is a simple computation involving the total magnetization and the spin and neighborhood magnetization of the node. In order to account for this operation, we include the term \textsc{NoFlipIS} in the above score. In the case of a spin flip, we must update one spin, the total magnetization and the magnetization of the neighborhoods that contain the spin, i.e., $2 + (a + b) \lambda_n$ updates on average. Considering also that we must check that the energy decreases, this explains the second term in \eqref{eq: score ising}.

On the other hand, the score used for BP is defined as:
\begin{equation*}
	\mathrm{\textsc{ScoreBP}} = \left[2 (1 - \eta)^2 (a + b) n \lambda_n + 2 (1 - \eta) n\right] \mathrm{\textsc{IterBP}},
\end{equation*}
where \textsc{IterBP} is the number of iterations until the target accuracy is reached. Each iteration involves two updates per directed edge and two updates per node, excluding the nodes with disclosed labels. On average, this amounts to $2 (1 - \eta)^2 (a + b) n \lambda_n + 2 (1 - \eta) n$ variables being updated on each iteration of the algorithm.

Table \ref{tab: comparison with belief propagation} summarizes a large number of experiments comparing Algorithm \ref{alg: algorithm} and BP in a logarithmic regime. For each value of $\eta$ and target accuracy, we ran each algorithm $10$ times, stopping when the target accuracy was reached, and then computed the average score over these experiments. The ratios of these averages, shown in the table, indicate that Algorithm \ref{alg: algorithm} is at least between $5$ and $7$ times faster. Figure \ref{fig: comparison with bp} plots the classification errors of Algorithm \ref{alg: algorithm} and BP as a function of the score.

\begin{figure}
	\centering
	\includegraphics{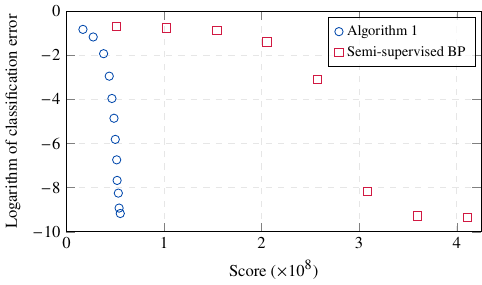}
	\caption{Experiments with $V_n^1 = V_n^2 = n = 500\ 000$, $\lambda_n = \log n$, $a = 3$, $b = 1$ and $\eta = 0.02$. The results for Algorithm \ref{alg: algorithm} correspond to $k 10^6$ iterations with $k \in \{1, \dots, 12\}$; each iteration updates at most one spin. The results for BP correspond to $k$ iterations with $k \in \{1, \dots, 8\}$; each iteration updates all variables. In both cases, the lowest classification error is roughly $0.01\%$.}
	\label{fig: comparison with bp}
\end{figure}

\subsection{Comparison with other semi-supervised methods}
\label{sub: performance evaluation}

Table \ref{tab: comparison with other algorithms} compares the performance of Algorithm \ref{alg: algorithm} with that of several semi-supervised clustering algorithms discussed in Section \ref{sec: introduction}, which we briefly describe below.
\begin{itemize}
	\item \emph{Asynchronous and Synchronous Consensus Dynamics.} The algorithm is initialized by assigning values $-1$ and $1$ to the nodes with revealed community labels, according to their respective labels, whereas the initial value is $0$ for all the other nodes. The nodes with unknown community labels iteratively replace their value by the average value across the neighboring nodes, and after multiple iterations, the community labels are estimated by looking at the sign of the value associated with each node. In the asynchronous version, the node that updates its value is selected uniformly at random, whereas all nodes are updated simultaneously in the synchronous version.
	
	\item \emph{Gossiping Algorithm.} The algorithm is initialized as above. On each iteration, an edge is selected uniformly at random and the average of the values in the endpoints of the edge is computed. If the community of a node in one endpoint of the edge was not revealed at the beginning of the algorithm, then the value of the node is replaced by the average mentioned earlier. After multiple iterations of the algorithm, the community labels are estimated using the signs of the values assigned to the nodes. As noted in Section \ref{sec: introduction}, this algorithm has been analyzed in \cite{xing2023almost}.
	
	\item \emph{Generalized Laplacian Methods.} Let $X(t)$ and $Y$ be $V_n \times 2$ matrices, such that $X(0)$ is full of zeros and the entries of $Y$ are zeros and ones indicating the communities of the nodes with revealed labels. On each iteration, these algorithms perform the update $X(t + 1) = \gamma D^{-\delta} A D^{\delta - 1} X(t) + (1 - \gamma) Y$, where $A$ is the adjacency matrix and $D$ is the diagonal matrix of degrees. In our experiments we set $\gamma = 0.95$ and $\delta \in \{0, 0.5, 1\}$, which correspond to PageRank, the Normalized Laplacian and the Standard Laplacian Methods, respectively. In \cite{avrachenkov2012generalized} these methods are shown to solve a suitable regularized optimization problem and have probabilistic interpretation in terms of random walks. After several iterations, the community labels are estimated by looking at the maximum value on each row of $X(t)$. 
	
	\item \emph{Poisson Learning.} The algorithm solves a Poisson equation that involves the graph Laplacian and has a right-hand side defined using the revealed community labels. The solution of the equation is estimated through an iterative procedure similar to that described above; see \cite[Algorithm 1]{calder2020poisson} for details.
\end{itemize}

For each row of Table \ref{tab: comparison with other algorithms}, we ran all algorithms $10$ times using the same graph for all algorithms on each run. Each algorithm was run for a number of iterations such that each node could be updated $20$ times on average. For Algorithm \ref{alg: algorithm} and the Asynchronous Consensus Dynamics, the maximum number of iterations was $20 V_n = 2 \times 10^5$ since a single node is updated on each iteration. However, Algorithm \ref{alg: algorithm} typically stopped much sooner, with the average number of spin flips ranging between $6548$ for $\eta = 0.1$ and $10445$ for $\eta = 0.01$, i.e., $0.65$ and $1.04$ flips per node, respectively. For the Synchronous Consensus Dynamics, the Generalized Laplacian Methods and Poisson Learning, we set the number of iterations to $20$ since all nodes are updated on each iteration. For Gossiping, the number of iterations was $20(a + b)V_n \lambda_n / 2 \simeq 3.4 \times 10^6$, since each node needs as many iterations as its degree to be paired with each neighbor and two nodes are updated on each iteration.

\renewcommand{\arraystretch}{0.5}
\begin{table}
	\centering
	\footnotesize
	
	\begin{tabular}{
			@{}
			l
			@{\hspace{10pt}}  
			c
			*{10}{@{\hspace{8pt}}l} 
			@{}
		}
		\toprule
		& & \multicolumn{10}{c}{$\eta$} \\
		\cmidrule(lr){3-12}
		& & {0.01} & {0.02} & {0.03} & {0.04} & {0.05} & {0.06} & {0.07} & {0.08} & {0.09} & {0.10} \\
		\midrule
		
		\multirow{2}{50mm}{Algorithm \ref{alg: algorithm} ($\alpha = 10$, $\beta_n = \infty$)} & $\mu$ & \bfseries 10.1 & \bfseries 0.14 & \bfseries 0.11 & \bfseries 0.13 & \bfseries 0.11 & \bfseries 0.12 & \bfseries 0.12 & \bfseries 0.12 & \bfseries 0.12 & \bfseries 0.12 \\
		& $\sigma$ & 29.9 & 0.03 & 0.03 & 0.05 & 0.03 & 0.04 & 0.04 & 0.02 & 0.03 & 0.04 \\
		\midrule
		
		\multirow{2}{*}{Asynchronous Consensus} & $\mu$ & 15.8 & 12.6 & 12.7 & 11.3 & 10.0 & 8.18 & 6.91 & 5.77 & 4.95 & 4.42 \\
		& $\sigma$ & 4.54 & 1.23 & 0.85 & 0.19 & 0.81 & 0.31 & 0.28 & 0.20 & 0.35 & 0.32 \\
		\midrule
		
		\multirow{2}{*}{Synchronous Consensus} & $\mu$ & 15.8 & 12.6 & 12.7 & 11.3 & 10.1 & 8.18 & 6.91 & 5.77 & 4.95 & 4.42 \\
		& $\sigma$ & 4.40 & 1.26 & 0.83 & 0.19 & 0.82 & 0.32 & 0.29 & 0.20 & 0.36 & 0.32 \\
		\midrule
		
		\multirow{2}{*}{Gossiping Algorithm} & $\mu$ & 44.4 & 39.5 & 36.4 & 34.6 & 33.5 & 32.0 & 30.9 & 29.9 & 29.1 & 28.6 \\
		& $\sigma$ & 1.84 & 1.69 & 0.69 & 0.57 & 0.64 & 0.32 & 0.72 & 0.42 & 0.46 & 0.86 \\
		\midrule
		
		\multirow{2}{*}{Standard Laplacian} & $\mu$ & 13.0 & 12.6 & 12.9 & 12.1 & 10.4 & 8.98 & 7.52 & 6.35 & 5.48 & 4.88 \\
		& $\sigma$ & 2.16 & 1.22 & 1.15 & 0.41 & 0.40 & 0.30 & 0.25 & 0.34 & 0.29 & 0.24 \\
		\midrule
		
		\multirow{2}{*}{Normalized Laplacian} & $\mu$ & 12.2 & 12.1 & 12.3 & 12.0 & 10.4 & 8.91 & 7.52 & 6.33 & 5.47 & 4.75 \\
		& $\sigma$ & 0.62 & 0.38 & 0.49 & 0.34 & 0.43 & 0.25 & 0.23 & 0.30 & 0.27 & 0.12 \\
		\midrule
		
		\multirow{2}{*}{PageRank Method} & $\mu$ & 12.1 & 12.4 & 12.3 & 12.0 & 10.4 & 8.92 & 7.59 & 6.46 & 5.57 & 4.79 \\
		& $\sigma$ & 0.47 & 0.34 & 0.37 & 0.29 & 0.45 & 0.27 & 0.24 & 0.30 & 0.24 & 0.17 \\
		\midrule
		
		\multirow{2}{*}{Poisson Learning} & $\mu$ & 11.6 & 12.0 & 11.9 & 11.4 & 9.58 & 8.10 & 6.85 & 5.74 & 4.93 & 4.26 \\
		& $\sigma$ & 0.44 & 0.32 & 0.39 & 0.28 & 0.44 & 0.24 & 0.26 & 0.28 & 0.22 & 0.15 \\
		\bottomrule
	\end{tabular}
	\caption{Percentage relative errors for $V_n^1 = V_n^2 = n = 5000$, $\lambda_n = \log n$, $a = 3$ and $b = 1$. The headers $\mu$ and $\sigma$ refer to the mean and standard deviation over $10$ experiments, respectively. For Algorithm \ref{alg: algorithm}, the hyperparameters are $\alpha = 10$ and $\beta_n = \infty$.}
	\label{tab: comparison with other algorithms}
\end{table}

Algorithm \ref{alg: algorithm} performed much better than the other algorithms, especially for $\eta \geq 0.02$, where its mean relative error is almost zero. For $\eta = 0.01$, the relative error of this algorithm was nearly zero in all the experiments except one, where it returned opposite community labels for nearly all the nodes; this explains the high standard deviation in this case. Except for the Gossiping Algorithm, which had by far the worst performance, the other algorithms had very similar mean relative errors for all values of $\eta \geq 0.02$, which in the best case were more than $4\%$ above the mean relative error obtained by Algorithm \ref{alg: algorithm}.

\section{Open problems}
\label{sec: conclusion}

We would like to conclude by mentioning a couple of open problems that seem relevant from a practical perspective and at the same time mathematically challenging.
\begin{enumerate}
	\item Our theoretical results prove that Algorithm \ref{alg: algorithm} is asymptotically effective provided that the hyperparameter $\alpha$ lies in the interval given by \eqref{eq: zero condition for mean-field limit}. However, our numerical results show that, for graphs of a given size, performance improves if $\alpha$ is selected within a suitable subinterval, as explained in Figure \ref{fig: sign plot}. Characterizing this behavior rigorously seems difficult and may require a nonasymptotic approach, different from the one adopted in the present paper.
	
	\item Our numerical results suggest that Algorithm \ref{alg: algorithm} could be effective in the unsupervised case as well, i.e., if it is initialized by sampling all the spins uniformly at random. In this case $\bz_n(0)$ vanishes as $n \to \infty$, but our simulations show that $\bz_n$ moves away from zero after some short fluctuations and then continues moving towards a state where the communities have opposite magnetizations. Proving this rigorously seems challenging and may require rescaling $\bz_n$ depending on how close it is to zero at the current time, with the behavior near zero being more similar to a diffusion.
\end{enumerate}

\newpage

\begin{appendices}
	
\section{Table of notation}
\label{app: table of notation}

\begin{center}
	\renewcommand{\arraystretch}{0.9}
	\begin{tabular}{c@{\hspace{10pt}}l@{\hspace{10pt}}}
		\multicolumn{2}{c}{Notation introduced in Section \ref{sec: problem formulation}} \\
		\hline
		$n$ & Scaling parameter, proportional to the number of nodes \\
		$\calV_n$, $V_n$ & Set and number of nodes in the graph, respectively \\
		$\calV_n^k$, $V_n^k$ & Set and number of nodes in community $k$, respectively \\
		$v^k$ & Limit of $V_n^k / n$ as $n \to \infty$ \\
		$\lambda_n$ & Scaling parameter for the average degree, e.g., $\lambda_n = \log n$ \\
		$p_n(k, l)$ & Edge probability between communities $k$ and $l$ \\
		$a_n$, $b_n$ & $p_n(k, k) = a_n = a \lambda_n / n$ and $p_n(k, l) = b_n = b \lambda_n / n$ if $k \neq l$ \\
		\hline
		\multicolumn{2}{c}{Notation introduced in Section \ref{sub: ising model}} \\
		\hline
		$u \sim v$ & Nodes $u$ and $v$ are connected by an edge \\
		$\Sigma_n$ & Set of configurations $\map{\sigma}{\calV_n}{\{-1, 1\}}$ \\
		$H_n(\sigma)$ & Hamiltonian associated with configuration $\sigma \in \Sigma_n$ \\
		$\alpha_n$ & Quadratic penalty factor $\alpha_n = \alpha \lambda_n / n$ in $H_n$ \\
		$\beta_n$ & Inverse temperature for Glauber dynamics \\
		\hline
		\multicolumn{2}{c}{Notation introduced in Section \ref{sub: glauber dynamics}} \\
		\hline
		$\Delta_n(\sigma, u)$ & Variation of the Hamiltonian $H_n(\sigma)$ when the spin of node $u$ flips \\
		$r\left(\beta_n \Delta_n(\sigma, u)\right)$ & Rate at which spin of node $u$ flips in configuration $\sigma$ for Glauber dynamics \\
		$\bsigma_n(t)$ & Configuration at time $t$ in continuous-time Glauber dynamics \\
		\hline
		\multicolumn{2}{c}{Notation introduced in Section \ref{sub: semi-supervised learning algorithm}} \\
		\hline
		$\eta$ & Fraction of nodes with disclosed community labels \\
		$t_{\mathrm{end}}$ & Running time of Glauber dynamics for a given target accuracy \\
		\hline
		\multicolumn{2}{c}{Notation introduced in Section \ref{sec: main results}} \\
		\hline
		$\calN_n(u)$, $N_n(u)$ & Set and number of nodes in neighborhood of node $u$ \\
		$\calN_n^k(u)$, $N_n^k(u)$ & Set and number of nodes in $\calN_n(u)$ belonging to community $k$ \\
		$\calY_s^k(\sigma)$, $Y_s^k(\sigma)$ & Set and number of nodes $u \in \calV_n^k$ with spin $\sigma(u) = s$ \\
		$\hat{\calY}_s^k(\sigma, u)$, $\hat{Y}_s^k(\sigma, u)$ & Set and number of nodes $v \in \calN_n^k(u)$ with spin $\sigma(v) = s$ \\
		$z^k(\sigma)$ & Normalized magnetization of community $k$ for configuration $\sigma$ \\
		$\bz_n^k(t)$ & Normalized magnetization of community $k$ for configuration $\bsigma_n(t)$ \\
		$t_n$ & Running time of Glauber dynamics for almost exact recovery \\
		\hline
		\multicolumn{2}{c}{Notation introduced in Section \ref{sec: mean-field approximation}} \\
		\hline
		$\bark$ & Complementary community index $\bark = 3 - k$, i.e., $\{k, \bark\} = \{1, 2\}$ \\
		$\calM_n^k(t)$ & Martingale error term in the mean-field approximation \\
		$\calE^k(\sigma)$ & Error term in the mean-field approximation due to graph topology \\
		$\calL^k(z)$ & $\calL^k(z) = (a - \alpha) v^k z^k + (b - \alpha) v^\bark z^\bark$, used in mean-field approximation \\
		\hline
		\multicolumn{2}{c}{Notation introduced in Section \ref{sec: mean-field limit}} \\
		\hline
		$\calA(\zeta, \xi)$ & Set $\set{z \in [-1 + \xi, 1 - \xi]^2}{\min \{\calL^1(z), -\calL^2(z)\} \geq \zeta}$ in Figure \ref{fig: attractivity region} \\
		$\Sigma_n(\zeta, \xi)$ & Set of configurations $\sigma$ such that $z(\sigma) \in \calA(\zeta, \xi)$
	\end{tabular}
\end{center}

\section{Tightness result}
\label{app: tightness}

In this section we prove Lemma \ref{lem: tightness of magnetization processes}, which requires the following standard result.

\begin{lemma} 
\label{lem: tightness criterion}
Let $\norm{\scdot}$ be a norm on $\R^2$. For each $T \geq 0$ and $h > 0$, we define the local modulus of continuity of a function $\bx \in D_{\R^2}[0, \infty)$ as
\begin{equation*}
w_T(\bx, h) \defeq \sup\set{\norm{\bx(t) - \bx(s)}}{s, t \in [0, T]\ \text{and}\ |t - s| \leq h}.
\end{equation*}
Consider a sequence of processes $\set{\bx_n}{n \geq 1}$ with sample paths in $D_{\R^2}[0, \infty)$ such that:
\begin{enumerate}
\item[(a)] $\bx_n(t)$ is tight in $\R^2$ for all $t$ in some dense subset of $[0, \infty)$,

\item[(b)] for all $T \geq 0$, we have
\begin{equation*}
\lim_{h \to 0} \limsup_{n \to \infty} E\left[\min\left\{w_T(\bx_n, h), 1\right\}\right] = 0.
\end{equation*}
\end{enumerate}
Then $\set{\bx_n}{n \geq 1}$ is tight in $D_{\R^2}[0, \infty)$ with respect to the topology of uniform convergence over compact sets. Also, every subsequence of $\set{\bx_n}{n \geq 1}$ has a further subsequence that converges weakly in $D_{\R^2}[0, \infty)$ to a process that is continuous with probability one.
\end{lemma}

\begin{proof}
For each $T \geq 0$ and $h > 0$, the modified modulus of continuity is defined as
\begin{equation*}
\tilde{w}_T(\bx, h) \defeq \inf_\calI \max_{I \in \calI} \sup_{s, t \in I} \norm{\bx(t) - \bx(s)} \quad \text{for all} \quad \bx \in D_{\R^2}[0, \infty),
\end{equation*}
where the infimum extends over all partitions $\calI$ of $[0, T]$ into subintervals $I = [s, t)$ such that $t - s \geq h$ if $t < T$. We observe that
\begin{equation*}
\tilde{w}_T(\bx, h) \leq w_T(\bx, h) \quad \text{for all} \quad \bx \in D_{\R^2}[0, \infty), \quad T \geq 0 \quad \text{and} \quad h > 0.
\end{equation*}
It follows from \cite[Theorems 23.8 and 23.9]{kallenberg1997foundations} that $\set{\bx_n}{n \geq 1}$ is tight in $D_{\R^2}[0, \infty)$ if (a) and (b) hold. By Prohorov's theorem, every subsequence of $\set{\bx_n}{n \geq 1}$ has a further subsequence that converges weakly in $D_{\R^2}[0, \infty)$, and \cite[Theorem 23.9]{kallenberg1997foundations} implies that the limiting process is almost surely continuous.
\end{proof}

We are now ready to prove Lemma \ref{lem: tightness of magnetization processes}.

\begin{proof}[Proof of Lemma \ref{lem: tightness of magnetization processes}]
It is clear that (a) of Lemma \ref{lem: tightness criterion} holds automatically because $\bz_n$ takes values in the compact set $[-1, 1]^2$. Hence, it suffices to check that (b) holds as well.

Note that the jump times of $\bz_n$ can be obtained from a thinning of a Poisson process $\calN_n$ with rate $V_n$ since the spins flip at rate less than one. Recall that $V_n^k / n \to v^k$ as $n \to \infty$. Thus, there exists $c > 0$ such that $V_n^k \geq cn$ for all $k \in \{1, 2\}$ and $n \geq 1$, and we have
\begin{equation*}
\norm{\bz_n(t) - \bz_n\left(t^-\right)}_1 = \left|\bz_n^k(t) - \bz_n^k\left(t^-\right)\right| = \frac{2}{V_n^k} \leq \frac{2}{c n}
\end{equation*}
whenever the spin of a node in community $V_n^k$ flips at time $t$. If we consider the modulus of continuity defined by the norm $\norm{\scdot}_1$, then we obtain
\begin{align*}
E\left[w_T\left(\bz_n, h\right)\right] &\leq 2 E\left[\sup_{t \in [0, T]} \left|\frac{\calN_n(t + h) - \calN_n(t)}{c n}\right|\right] \\
&\leq 2 E\left[\sup_{t \in [0, T]} \left|\frac{\calN_n(t + h) - V_n(t + h)}{c n}\right|\right] + 2 E\left[\sup_{t \in [0, T]} \left|\frac{\calN_n(t) - V_nt}{c n}\right|\right] + \frac{2h V_n}{c n}.
\end{align*}

By Doob's submartingale inequality moment version,
\begin{equation*}
E\left[\sup_{t \in [0, T]} \left|\frac{\calN_n(t) - V_n t}{c n}\right|^2\right] \leq \frac{4 V_n T}{c^2 n^2}.
\end{equation*}
Then it follows from Jensen's inequality that
\begin{equation*}
E\left[w_T\left(\bz_n, h\right)\right] \leq 2 \sqrt{\frac{4 V_n (T + h)}{c^2 n^2}} + 2 \sqrt{\frac{4 V_n T}{c^2 n^2}} + \frac{2 h V_n}{c n}.
\end{equation*}
Note that $V_n / n \to v^1 + v^2$ as $n \to \infty$. Hence, we conclude that property (b) holds:
\begin{equation*}
\lim_{h \to 0} \limsup_{n \to \infty} E\left[\min\{w_T\left(\bz_n, 1\right), 1\}\right] \leq \lim_{h \to 0} \limsup_{n \to \infty} E\left[w_T\left(\bz_n, h\right)\right] = 0,
\end{equation*}
and we complete the proof by invoking Lemma \ref{lem: tightness criterion}.
\end{proof}

\section{Concentration inequalities}
\label{app: concentration inequalities}

The following lemma contains some Chernoff bounds for the binomial distribution, which we will use to prove Lemmas \ref{lem: concentration of intracommunity degrees} and \ref{lem: concentration of signed degrees} below.

\begin{lemma}
\label{lem: concentration inequalities for binomial distribution}
If $X \sim \Bin(n, p)$ is binomially distributed with mean $\mu \defeq n p$, then
\begin{align*}
&\condp*{X \geq \mu + x} \leq \e^{-\mu \varphi\left(\frac{x}{\mu}\right)} \quad \text{for all} \quad x \geq 0, \\
&\condp*{\left|X - \mu\right| \geq x \mu} \leq 2 \e^{-\frac{x^2\mu}{3}} \quad \text{for all} \quad x \in \left(0, 3/2\right],
\end{align*}
where $\varphi(x) \defeq (1 + x)\log(1 + x) - x$ for all $x \geq 0$.
\end{lemma}

\begin{proof}
The inequalities are proved in \cite[Theorem 2.1 and Corollary 2.3]{janson2011random}, respectively.
\end{proof}

The following lemma bounds the fraction of nodes $u \in \calV_n^k$ such that the number of neighbors that $u$ has in $\calV_n^k$ deviates too much from the expected value.

\begin{lemma}
\label{lem: concentration of intracommunity degrees}
Given $\delta_n > 0$, recall that
\begin{equation*}
A_n^k(u) \defeq \left\{\left|N_n^k(u) - a_n V_n^k\right| \geq \sqrt{a_n V_n^k \delta_n}\right\} \quad \text{for all} \quad u \in \calV_n^k.
\end{equation*}
Suppose that the sequences $\gamma_n > 0$ and $\delta_n$ are such that
\begin{equation*}
\lim_{n \to \infty} \delta_n = \infty, \quad \limsup_{n \to \infty} \frac{\delta_n}{\lambda_n} < \infty \quad \text{and} \quad \limsup_{n \to \infty} \frac{\gamma_n}{\delta_n} < \infty.
\end{equation*}
Then the following limit holds:
\begin{equation*}
\frac{\gamma_n}{n} \sum_{u \in \calV_n^k} \ind{A_n^k(u)} \Rightarrow 0 \quad \text{as} \quad n \to \infty \quad \text{for all} \quad k \in \{1, 2\}.
\end{equation*}
\end{lemma}

\begin{proof}
Given $\varepsilon > 0$, Markov's inequality yields
\begin{align*}
P\left(\frac{\gamma_n}{n} \sum_{u \in \calV_n^k} \ind{A_n^k(u)} \geq \varepsilon\right) &\leq \left(\frac{\gamma_n}{\varepsilon n}\right)^2 E\left[\sum_{u, v \in \calV_n^k} \ind{A_n^k(u)} \ind{A_n^k(v)}\right] \\
&= \left(\frac{\gamma_n}{\varepsilon n}\right)^2\left[V_n^k\condp*{A_n^k(u)} + V_n^k\left(V_n^k - 1\right)\condp*{A_n^k(u) \cap A_n^k(v)}\right],
\end{align*}
where the nodes $u$ and $v$ in the last line satisfy $u, v \in \calV_n^k$ and $u \neq v$. Therefore, it suffices to prove that the right-hand side approaches zero as $n \to \infty$ for each $\varepsilon > 0$.

Since $\delta_n / \lambda_n$ is bounded, there exists $c \in (0, 1)$ such that
\begin{equation*}
\sqrt{\frac{c \delta_n}{a_n\left(V_n^k - i\right)}} \in \left(0, 3/2\right] \quad \text{for all} \quad i \in \{1, 2\} \quad \text{and all large enough}\ n.
\end{equation*}
Furthermore, let us consider all the intervals of the form
\begin{equation*}
\left(a_n \left(V_n^k - i\right) + j - \sqrt{a_n \left(V_n^k - i\right) c \delta_n}, a_n \left(V_n^k - i\right) + j + \sqrt{a_n \left(V_n^k - i\right) c \delta_n}\right)
\end{equation*}
with $i \in \{1, 2\}$ and $j \in \{0, 1\}$. These intervals are contained in the interval
\begin{equation*}
\left(a_n V_n^k - \sqrt{a_n V_n^k \delta_n}, a_n V_n^k + \sqrt{a_n V_n^k \delta_n}\right) \quad \text{for all large enough}\ n.
\end{equation*}

Using the above observations, we conclude from Lemma \ref{lem: concentration inequalities for binomial distribution} that
\begin{equation*}
P\left(A_n^k(u)\right) \leq P\left(\left|N_n^k(u) - a_n(V_n^k - 1)\right| \geq \sqrt{a_n (V_n^k - 1) c \delta_n}\right) \leq 2 \e^{-\frac{c \delta_n}{3}}
\end{equation*}
for all sufficiently large $n$, because $N_n^k(u) \sim \Bin\left(V_n^k - 1, a_n\right)$.

In addition, we observe that
\begin{align*}
P\left(A_n^k(u) \cap A_n^k(v)\right) &= \condp*{A_n^k(u) \cap A_n^k(v) | u \nsim v} (1 - a_n) + \condp*{A_n^k(u) \cap A_n^k(v) | u \sim v} a_n \\
&=\left[\condp*{A_n^k(u) | u \nsim v}\right]^2(1 - a_n) + \left[\condp*{A_n^k(u) | u \sim v}\right]^2a_n.
\end{align*}
Applying Lemma \ref{lem: concentration inequalities for binomial distribution} again, we obtain
\begin{align*}
&\condp*{A_n^k(u) | u \nsim v} \leq \condp*{\left|N_n^k(u) - a_n\left(V_n^k - 2\right)\right| \geq \sqrt{a_n \left(V_n^k - 2\right) c \delta_n} | u \nsim v} \leq 2 \e^{-\frac{c \delta_n}{3}}, \\
&\condp*{A_n^k(u) | u \sim v} \leq \condp*{\left|N_n^k(u) - 1 - a_n\left(V_n^k - 2\right)\right| \geq \sqrt{a_n \left(V_n^k - 2\right) c \delta_n} | u \sim v} \leq 2 \e^{-\frac{c \delta_n}{3}},
\end{align*}
for all large enough $n$. In the former case, note that $N_n^k(u) \sim \Bin\left(V_n^k - 2, a_n\right)$ when $u \nsim v$, whereas $N_n^k(u) - 1 \sim \Bin\left(V_n^k - 2, a_n\right)$ in the latter case where $u \sim v$.

Putting the above inequalities together, we get
\begin{equation*}
P\left(\frac{\gamma_n}{n} \sum_{u \in \calV_n^k} \ind{A_n^k(u)} \geq \varepsilon\right) \leq \left(\frac{\gamma_n}{\varepsilon n}\right)^2 \left[2V_n^k\e^{-\frac{c \delta_n}{3}} + 4V_n^k\left(V_n^k - 1\right)\e^{-\frac{2 c \delta_n}{3}}\right]
\end{equation*}
for all sufficiently large $n$. Because $\gamma_n / \delta_n$ is bounded and $\delta_n \to \infty$ as $n \to \infty$, we conclude that the right-hand side vanishes and thus complete the proof.
\end{proof}

The next lemma is similar to Lemma \ref{lem: concentration of intracommunity degrees} in spirit but takes into account the spin of the nodes.

\begin{lemma}
\label{lem: concentration of signed degrees}
Given $\delta_n > 0$, recall that
\begin{equation*}
B_s^{k, l}(\sigma, u) \defeq \left\{\left|\hat{Y}_s^l(\sigma, u) - p_n(k, l) Y_s^l(\sigma)\right| \geq \sqrt{p_n(k, l) Y_s^l(\sigma) \delta_n}\right\}
\end{equation*}
for all $\sigma \in \Sigma_n$ and $u \in \calV_n^k$. Suppose that $\gamma_n, \delta_n > 0$ and $\xi_n \in (0, 1)$ are such that
\begin{equation*}
\lim_{n \to \infty} \delta_n = \infty, \quad \lim_{n \to \infty} \frac{\delta_n}{\xi_n\lambda_n} = 0 \quad \text{and} \quad \lim_{n \to \infty} \frac{\delta_n}{\gamma_n} = \infty.
\end{equation*}
If we fix $\zeta > 0$ and $(k, s) \neq (l, t)$, then we have
\begin{equation*}
\max_{\sigma \in \Sigma_n(\zeta, \xi_n)} \frac{\gamma_n}{n}\sum_{u \in \calY_s^k(\sigma)} \ind{B_t^{k, l}(\sigma, u)} \Rightarrow 0 \quad \text{as} \quad n \to \infty.
\end{equation*}
\end{lemma}

\begin{proof}
If $\sigma \in \Sigma_n(\zeta, \xi_n)$, then $Y_t^l(\sigma) \geq \xi_n V_n^l / 2$. Hence, for all sufficiently large $n$,
\begin{equation*}
0 < \sqrt{\frac{\delta_n}{p_n(k, l) Y_s^l(\sigma)}} \leq \sqrt{\frac{2\delta_n}{\xi_n p_n(k, l) V_n^l}} \leq \sqrt{\frac{2\delta_n n}{b\xi_n\lambda_nV_n^l}} \leq \frac{3}{2}.
\end{equation*}
Since $\hat{Y}_t^l(\sigma, u) \sim \Bin\left(Y_t^l(\sigma), p_n(k, l)\right)$ for all $u \in \calY_s^k(\sigma)$, Lemma \ref{lem: concentration inequalities for binomial distribution} implies that
\begin{equation*}
\condp*{B_t^{k, l}(\sigma, u)} \leq 2\e^{-\frac{\delta_n}{3}} \quad \text{for all} \quad \sigma \in \Sigma_n(\zeta, \xi_n) \quad u \in \calY_s^k(\sigma) \quad \text{and all large enough}\ n.
\end{equation*}
Furthermore, the random variables $\set{\hat{Y}_t^l(\sigma, u)}{u \in \calY_s^k(\sigma)}$ are independent if $\sigma$ is fixed since $(k, s) \neq (l, t)$. Thus, for all large enough $n$, the next stochastic inequality holds:
\begin{equation*}
\sum_{u \in \calY_s^k(\sigma)} \ind{B_t^{k, l}(\sigma, u)} \leq_{st} \Bin\left(Y_s^k(\sigma), 2\e^{-\frac{\delta_n}{3}}\right) \quad \text{for all} \quad \sigma \in \Sigma_n(\zeta, \xi_n).
\end{equation*}

Fix $\varepsilon > 0$ and let us write $\mu \defeq 2Y_s^k(\sigma)\e^{-\frac{\delta_n}{3}}$ and $x \defeq \varepsilon n / \gamma_n - \mu$ for brevity. It follows from the above stochastic inequality and Lemma \ref{lem: concentration inequalities for binomial distribution} that
\begin{equation*}
\condp*{\frac{\gamma_n}{n}\sum_{u \in \calY_s^k(\sigma)} \ind{B_t^{k, l}(\sigma, u)} \geq \varepsilon} \leq \e^{-\mu\varphi\left(\frac{x}{\mu}\right)} \quad \text{for all} \quad \sigma \in \Sigma_n(\zeta, \xi_n)
\end{equation*}
and all large enough $n$, where $\varphi$ is defined as in the statement of Lemma \ref{lem: concentration inequalities for binomial distribution}. Now note that
\begin{align*}
\mu\varphi\left(\frac{x}{\mu}\right) = \mu \varphi\left(\frac{\varepsilon n}{\mu \gamma_n} - 1\right) &= \frac{\varepsilon n}{\gamma_n} \left[\log\left(\frac{\varepsilon n}{\mu \gamma_n}\right) - 1\right] + \mu \\
&\geq \frac{\varepsilon n}{\gamma_n} \left[\log\left(\frac{\varepsilon n \e^{\frac{\delta_n}{3}}}{2 V_n^k \gamma_n}\right) - 1\right] = \varepsilon n \left[\frac{\delta_n}{3 \gamma_n} -\frac{1}{\gamma_n} \log\left(\frac{2 V_n^k \gamma_n}{\varepsilon n}\right) - \frac{1}{\gamma_n}\right]
\end{align*}
for all $\sigma \in \Sigma_n(\zeta, \xi_n)$, where we used that $Y_s^k(\sigma) \leq V_n^k$. Hence, for all sufficiently large $n$,
\begin{equation*}
\condp*{\max_{\sigma \in \Sigma_n(\zeta, \xi_n)} \frac{\gamma_n}{n} \sum_{u \in \calY_s^k(\sigma)} \ind{B_t^{k, l}(\sigma, u)} \geq \varepsilon} \leq 2^{V_n} \e^{-\varepsilon n \left[\frac{\delta_n}{3\gamma_n} -\frac{1}{\gamma_n} \log\left(\frac{2 V_n^k \gamma_n}{\varepsilon n}\right) - \frac{1}{\gamma_n}\right]},
\end{equation*}
and the right-hand side approaches zero as $n \to \infty$ because $\delta_n / \gamma_n \to \infty$.
\end{proof}

\section{Other intermediate results}
\label{app: other intermediate results}

\begin{proof}[Proof of Proposition \ref{prop: maximum likelihood estimator}]
Let $A$ denote the adjacency matrix of $\calG$. Then
\begin{equation}
p_\calG(\sigma) = \prod_{u < v} \left[a_n^{A_{uv}}(1 - a_n)^{1 - A_{uv}}\ind{\sigma(u) = \sigma(v)} + b_n^{A_{uv}}(1 - b_n)^{1 - A_{uv}}\ind{\sigma(u) \neq \sigma(v)}\right],
\end{equation}
where we have identified $\calV_n$ with $\{1, \dots, V_n\}$. It follows that
\begin{equation}
\label{eq: log-likelihood}
\begin{split}
\log p_\calG(\sigma) &= \sum_{u < v} \left[A_{uv}\log a_n + (1 - A_{uv})\log(1 - a_n)\right] \ind{\sigma(u) = \sigma(v)} \\
&+ \sum_{u < v} \left[A_{uv}\log b_n + (1 - A_{uv})\log(1 - b_n)\right] \ind{\sigma(u) \neq \sigma(v)} \\
&= \frac{1}{2}\sum_{u < v} \left[A_{uv}\log\left(\frac{a_n}{b_n}\right) + (1 - A_{uv})\log\left(\frac{1 - a_n}{1 - b_n}\right)\right]\sigma(u)\sigma(v) \\
&+ \frac{1}{2}\sum_{u < v} \left[A_{uv}\log(a_n b_n) + (1 - A_{uv})\log\left((1 - a_n)(1 - b_n)\right)\right],
\end{split}
\end{equation}
where we used that $\ind{\sigma(u) = \sigma(v)} = [1 + \sigma(u)\sigma(v)] / 2$ and $\ind{\sigma(u) \neq \sigma(v)} = [1 - \sigma(u)\sigma(v)] / 2$.

Note that $p_\calG$ and $\log p_\calG$ have the same maximizers. Moreover, the second term on the right-hand side of \eqref{eq: log-likelihood} does not depend on $\sigma$, so we may focus on maximizing the first term. This term is equal to
\begin{equation*}
\frac{1}{2}\log\left(\frac{b_n(1 - a_n)}{a_n(1 - b_n)}\right)\left[-\sum_{u < v} A_{uv}\sigma(u)\sigma(v) + \rho_n\sum_{u < v} \sigma(u)\sigma(v)\right].
\end{equation*}
Furthermore, observe that
\begin{align*}
&\sum_{u < v} A_{uv}\sigma(u)\sigma(v) = \frac{1}{2}\sum_{u \sim v} \sigma(u)\sigma(v) \quad \text{and} \quad \sum_{u < v} \sigma(u)\sigma(v) = \frac{1}{2} \left[\sum_{u \in \calV_n} \sigma(u)\right]^2 - \frac{1}{2}V_n.
\end{align*}
The second term on the right-hand side of the last equation does not depend on $\sigma$. Hence, we conclude that finding the maximizers of $p_\calG$ is equivalent to finding the maximizers of
\begin{equation*}
\frac{1}{2}\log\left(\frac{b_n(1 - a_n)}{a_n(1 - b_n)}\right)\left[-\frac{1}{2}\sum_{u \sim v} \sigma(u)\sigma(v) + \frac{\rho_n}{2}\left[\sum_{u \in \calV_n} \sigma(u)\right]^2\right].
\end{equation*}
Moreover, the maximizers of the latter expression are the minimizers of
\begin{equation*}
-\frac{1}{2}\sum_{u \sim v} \sigma(u)\sigma(v) + \frac{\rho_n}{2}\left[\sum_{u \in \calV_n} \sigma(u)\right]^2
\end{equation*}
because $a > b > 0$ implies that $b_n(1 - a_n) < a_n(1 - b_n)$.
\end{proof}

The following lemma proves \eqref{seq3: weak limits}.

\begin{lemma}
\label{lem: initial conditions}
If $\gamma_n / \sqrt{n} \to 0$ as $n \to \infty$, then $\gamma_n \left[\bz_n(0) - \bz_\infty(0)\right] \Rightarrow 0$ as $n \to \infty$.
\end{lemma}

\begin{proof}
The construction in Algorithm \ref{alg: algorithm} implies that
\begin{equation*}
\bz_n^k(0) = \frac{1}{V_n^k} \sum_{u \in \calV_n^k} I_n^k(u),
\end{equation*}
where the random variables $I_n^k(u)$ are mutually independent and such that
\begin{equation*}
\condp*{I_n^k(u) = 1} = \eta\ind{k = 1} + \frac{1 - \eta}{2} \quad \text{and} \quad \condp*{I_n^k(u) = - 1} = \eta\ind{k = 2} + \frac{1 - \eta}{2}.
\end{equation*}
By the central limit theorem, $\sqrt{n}\left[\bz_n(0) - \bz_\infty(0)\right]$ converges weakly to a bivariate normal distribution as $n \to \infty$. As a result, we conclude that $\gamma_n\left[\bz_n(0) - \bz_\infty(0)\right] \Rightarrow 0$.
\end{proof}

Next we provide the proof of Lemma \ref{lem: representation lemma}.

\begin{proof}[Proof of Lemma \ref{lem: representation lemma}]
Let $S_{\R^2}[0, T]$ be the space of c\`adl\`ag functions from $[0, T]$ into $\R^2$ with the Skorohod-$J_1$ topology. The first two limits in \eqref{eq: weak limits} hold with respect to the topology of uniform convergence over compact sets and thus also with respect to the Skorohod-$J_1$ topology. Because $S_{\R^2}[0, T]$ is separable, the product $S_{\R^2}[0, T] \times S_{\R^2}[0, T] \times \R^2 \times \R^2$ is separable with respect to the product topology. Therefore, the process
\begin{equation*}
\left(\bz_m, \gamma_m\left(\frac{\calM_m^1}{V_m^1}, \frac{\calM_m^2}{V_m^2}\right), \gamma_m\left[\bz_m(0) - \bz_\infty(0)\right], \gamma_m\left(\max_{\sigma \in \Sigma_m\left(\theta, \theta\right)} \calE^1(\sigma), \max_{\sigma \in \Sigma_m\left(\theta, \theta\right)} \calE^2(\sigma)\right)\right)
\end{equation*}
is measurable with respect to the Borel $\sigma$-algebra associated with the product topology; see \cite[Appendix M10]{billingsley2013convergence}. It follows from \cite[Theorem 3.1]{billingsley2013convergence} that the above process converges weakly to $(\bz, 0, 0, 0)$ as $m \to \infty$ because the right-hand sides of the last three limits in \eqref{eq: weak limits} are deterministic processes; see \cite[Lemma 9]{goldsztajn2023load} for further details.

It follows from Skorohod's representation theorem that \eqref{eq: strong limits} holds but considering the Skorohod-$J_1$ topology for the first two limits, instead of the topology of uniform convergence over compact sets. However, the limiting processes in \eqref{seq1: strong limits} and \eqref{seq2: strong limits} are almost surely continuous, so the limits hold with respect to the latter topology as well.
\end{proof}

\end{appendices}

\newcommand{\noop}[1]{}
\bibliographystyle{IEEEtranS}
\bibliography{abbreviated}

\end{document}